\documentclass[11pt,a4paper,reqno]{amsart}
\usepackage{amsmath,amssymb,amsfonts,amsthm}
\usepackage{dsfont}

\usepackage{a4wide}
\usepackage{color}
\usepackage{pdfsync}
\usepackage{graphicx,subfigure}
\usepackage{hyperref}

\usepackage{enumitem}
\makeatletter
\def\namedlabel#1#2{\begingroup
 #2%
 \def\@currentlabel{#2}%
 \phantomsection\label{#1}\endgroup
}
\makeatother

\usepackage{dsfont}

\newtheorem{theorem}{Theorem}[section]

\newtheorem{proposition}[theorem]{Proposition}

\newtheorem{remark}[theorem]{Remark}
\numberwithin{equation}{section}
\makeatletter

\@addtoreset{equation}{section}
\makeatother

\newcommand{\Z}{{\mathbb Z}}

\newcommand{\R}{{\mathbb R}}

\newcommand{\K}{{\mathbb C}}
\newcommand{\eps}{{\varepsilon}}
\def\cal#1{\mathcal{#1}}

\def\K{\cal K}
\def\Tc{\mathcal{T}}
\def\Tca#1{{\Tc^{\eps}_{#1}}}

\def\curl{{\rm curl }\,}
\def\div{{\rm div }\,}
\def\loc{\, {\rm loc}}
\def\supp{{{\rm supp}\;}}

\title[Porous medium with irregular holes]{A note on the regularity of the holes for permeability property through a perforated domain for the 2D Euler equations\\
{\it \tiny Dedicated to Professor Jean-Yves Chemin on the Occasion of his {\rm 60}th Birthday}}

\author[C. Lacave \& C. Wang]{Christophe Lacave \& Chao Wang}

\address[C. Lacave]{Univ. Grenoble Alpes, CNRS, Institut Fourier, F-38000 Grenoble, France.}
\email{Christophe.Lacave@univ-grenoble-alpes.fr}

\address[C. Wang]{School of Mathematical Sciences, Peking University\\
 Beijing 100871, P. R. China.}
\email{wangchao@math.pku.edu.cn}
\date{\today}

\begin{document}
\maketitle

\begin{abstract}
For equations of order two with the Dirichlet boundary condition, as the Laplace problem, the Stokes and the Navier-Stokes systems, perforated domains were only studied when the distance between the holes $d_{\varepsilon}$ is equal or much larger than the size of the holes $\varepsilon$. Such a diluted porous medium is interesting because it contains some cases where we have a non-negligible effect on the solution when $(\varepsilon,d_{\varepsilon})\to (0,0)$. Smaller distance was avoided for mathematical reasons and for theses large distances, the geometry of the holes does not affect -or few- the asymptotic result. Very recently, it was shown for the 2D-Euler equations that a porous medium is non-negligible only for inter-holes distance much smaller than the size of the holes. For this result, the regularity of holes boundary plays a crucial role, and the permeability criterium depends on the geometry of the lateral boundary. In this paper, we relax slightly the regularity condition, allowing a corner, and we note that a line of irregular obstacles cannot slow down a perfect fluid in any regime such that $\varepsilon \ln d_{\varepsilon} \to 0$.
\end{abstract}


\section{Introduction}
In this article, we consider the behavior of the 2-D Euler equations in a porous medium. The velocity $u^\eps=(u_1^\eps, u_2^\eps)$ of an ideal incompressible fluid filling a domain $\Omega^\eps$ is governed by the Euler equations:
\begin{equation}\label{Euler}
\left\{
\begin{split}
&\partial_t u^\eps+u^\eps \cdot \nabla u^\eps +\nabla p^\eps=0,\quad (t,x)\in (0,\infty)\times \Omega^\eps;\\
&\div u^\eps=0,\quad (t,x)\in [0,\infty)\times \Omega^\eps;\\
& u^\eps\cdot n=0,\quad (t,x)\in [0,\infty)\times \partial\Omega^\eps;\\
& u^\eps(0,\cdot)=u^\eps_0,\quad x\in \Omega^\eps,
\end{split}
\right.
\end{equation}
where $p^\eps$ is the pressure and $\Omega^\eps$ is an exterior domain that is defined later.

Since these equations were established by Euler in 1755, the study of well-posedness and stability was a constant concern (see e.g. the references given in the introduction of \cite{GV-L}). One of the main reasons of so large literature is that the vorticity, defined by $\omega^\eps:=\curl u^\varepsilon=\partial_1u^\eps_2-\partial_2 u^\eps_1$, satisfies a transport equation
\begin{equation}\label{transport}
\partial_t \omega^\eps+ u^\eps\cdot \nabla \omega^\eps=0,\quad (t,x) \in (0,\infty)\times\Omega^\eps,
\end{equation}
where the velocity should be reconstructed from $\omega^\varepsilon$. This special structure allows to place the 2D-Euler equations in the intersection of many mathematical areas: non-linear PDE's, transport equation with flow maps, elliptic problem with Green kernel and conformal mapping (i.e. using the tools from complex analysis), geodesic flows on a Riemannian manifold, convex integration...

Here, we are interested by the influence of a porous medium on the behavior of the perfect fluid. The porous medium is modelized by $N_{\varepsilon}$ impermeable obstacles (also called inclusions or holes), of size $\varepsilon$ and separated by a distance $d_{\varepsilon}$. For practical interest, it is important to understand the leading behavior when $\varepsilon$ and $d_{\varepsilon}$ are very small compared to the experiment scale.

Such a question is a standard issue in the homogenization problems. For the Laplace equation, it is easy to show that the perforated domain has no effect at the limit in the regime $d_{\varepsilon} \sqrt{|\ln \varepsilon|}\to \infty$ if the holes are uniformly distributed on a surface and in the regime $d_{\varepsilon} |\ln \varepsilon|\to \infty$ if the holes are distributed on a curve. If the above quantities tend to $C>0$ instead to $\infty$, it was proved that we get an homogenized system for the Laplace, Stokes and Navier-Stokes steady flows \cite{Allaire, Allaire2, Mikelic, SP, Tartar}. Therein, the authors then considered very diluted porous medium-$d_{\varepsilon}\gg \varepsilon^\beta$ for any $\beta \in (0,1]$-and it is natural that the criterium does not depend on the geometry of the holes. For inviscid fluids, \cite{Lions-M, MP} obtain an homogenized limit for the weakly nonlinear Euler flow through a periodic grid, i.e. in the regime $d_{\varepsilon}=\varepsilon$ when the holes are distributed on a surface. The Euler equations were treated in \cite{BLM, LM}, where the cases of inter-holes distances smaller than the hole sizes were finally achieved. Here, the geometry of the lateral boundaries of the holes plays a role in the criterium, so we define precisely the domain properties.

The shape of the inclusions $\K$ was assumed to be a simply-connected compact subset of $[-1,1]^2$ such that $\partial K\in C^{1,\alpha}$ for $\alpha >0$ is a Jordan curve. All the inclusions considered have the same shape:
\begin{equation}\label{domain1}
\K_{i,j}^{\eps}:= z_{i,j}^{\varepsilon} + \tfrac\eps 2 \K,
\end{equation}
where the points $z_{i,j}^{\varepsilon}\in \R^2$ are uniformly distributed such that the inclusions of size $\varepsilon$ are at least separated by a distance $d_\varepsilon$, i.e. for $i,j\in \Z$ and $\varepsilon>0$, we set
\begin{equation}\label{domain2}
z_{i,j}^{\varepsilon}:=(\tfrac\eps 2+(i-1)(\eps+d_{\varepsilon}),(j-1)(\eps+d_\eps))=(\tfrac\eps 2,0)+(\eps+d_{\eps})(i-1,j-1).
\end{equation}
In the horizontal direction, we consider the maximal number of inclusions that we can distribute on the unit segment $[0,1]$, hence we consider
$$i=1,\dots, N_{\varepsilon} \text{ in \eqref{domain1}-\eqref{domain2}},$$
with
$$N_{\varepsilon}=\left[\frac{1+d_\eps}{\eps+d_\eps}\right]$$
(where $[x]$ denotes the integer part of $x$). In the vertical direction, we consider two situations:
\begin{itemize}
 \item inclusions covering the unit square, namely
 $$j=1,\dots, N_{\varepsilon} \text{ in \eqref{domain1}-\eqref{domain2}};$$
 \item inclusions concentrated on the unit segment, namely
 $$j=1 \text{ in \eqref{domain1}-\eqref{domain2}}.$$
\end{itemize}

When the obstacles are distributed only in one direction, we need to describe the geometry of the lateral boundaries around the points where the distances between two holes are reached. For simplicity, let us assume that $(\pm 1,0)\in \partial \K$ and that the boundary $\partial\K$ is locally parametrized around $(1,0)$ by
\[
x(s) = 1- \rho |s|^{1+\gamma}, \quad y(s)=s,\quad s\in [-\delta,\delta]
\]
with $\rho ,\gamma>0$, where $\gamma$ is called the tangency exponent, and $\delta>0$ can be assumed small.
For instance, $\gamma =1$ corresponds formally to the case where $\K$ is the unit ball\footnote{or any regular compact set whose the curvature is non zero and finite, see (H2) in \cite{LM} for the extension of the tangency exponent to any boundary.}, whereas $\gamma=\infty$ corresponds to the case where the solid is flat near $(\pm 1,0)$: $[( 1,-\rho_{0}),( 1,\rho_{0})]\subset \partial \K$.

In the case of holes distributed in one direction, the main results of \cite{BLM, LM} are as follows: if $\frac{d_{\varepsilon}}{\varepsilon^{2+\frac1\gamma}}\to \infty$ then the limit motion is not perturbed by the porous medium, whereas, if $\frac{d_{\varepsilon}}{\varepsilon^{2+\frac1\gamma}}\to 0$ then the unit segment becomes impermeable at the limit. We note here that the regimes considered correspond to close inclusions $d_{\varepsilon}\ll \varepsilon$, and it is physically natural that a fluid passes easier between disks than between flat solids.

 In the case of holes distributed in the two directions, the asymptotic behavior depends on the limit of $\frac{d_{\varepsilon}}{\varepsilon}$: if this limit is $\infty$ the presence of the porous medium is not felt at the limit, whereas the unit square becomes impermeable if the limit is zero. Even if the criterium is independent of $\gamma$, it was important in the analysis to have the existence of such a $\gamma>0$, as a consequence of the $C^{1,\alpha}$ regularity assumption.

In both results, it was crucial that $\gamma>0$, and the case of a corner-which corresponds to $\gamma=0$-was one of the open problems listed in \cite{LM}. More precisely, with a corner, the following questions are unsolved:\label{openquestions}
\begin{enumerate}
 \item if the inclusions are distributed in one direction, are there some regimes such that the porous medium is not felt at the limit ?
 \item if the inclusions are distributed in one direction, are there some regimes when we observe the impermeable segment ?
 \item if the inclusions are distributed in two directions, is it possible to state the impermeability result when $\frac{d_{\varepsilon}}{\varepsilon}\to 0$ ?
\end{enumerate}

As we will explain in the final remark, the questions (2) and (3) are unreachable without changing the full analysis. In this paper, we focus on the first question, where we use the technics developed in \cite{LMW} by the authors, in particular the precise behavior of the conformal mapping in the neighborhood of corners. Formally, taking $\gamma=0$ in the criterium $d_{\varepsilon}\gg \varepsilon^{2+\frac1\gamma}$, we would like to prove that the fluid is not perturbed by the porous medium if $d_{\varepsilon}=\varepsilon^\beta$ for any $\beta>0$ arbitrary large.

More precisely, we assume in this article that
\begin{description}
\item[\namedlabel{H1}{\rm\bf(H1)}] $\partial \K$ is a Jordan curve of class $C^{1,\alpha}$ for $\alpha >0$ except in a finite number of points $\{ x_{k}\}_{k=1,\dots,N}$ where $\partial \Omega$ is a corner of angle $\theta_{k}$,
\end{description}
 which reads as
\[ \lim_{s\to 0,s>0} {\rm Angle}(-\Gamma'(s_k-s),\Gamma'(s_k+s))=\theta_{k}\in [0,2\pi]\]
where $\Gamma$ is a parametrization of $\partial \Omega$ (couterclockwise direction) and the points $\{ x_{k}\}_{k=1,\dots,N}$ are of parameter $\{ s_{k}\}_{k=1,\dots,N}$. With this definition, $\theta_{k}$ corresponds to the angle in the fluid, which means that $\theta_{k} =3\pi/2$ if $\K$ is a square. Moreover, we assume that a corner is exactly located at the point where the distance is reached between $\K_{i,1}^\varepsilon$ and $\K_{i+1,1}^\varepsilon$, for instance, let us assume that
\begin{description}
\item[\namedlabel{H2}{\rm\bf(H2)}] $x_{1}=(1,0)$ and there exists $\rho>0$ such that the set $\K\cap(\R^+\times \{s\}) \subset [0,(1-\rho |s|)]\times\{s\}$ for all $s\in [-1,1]$.
\end{description}
Here, we have assumed that there is a corner at the point $(1,0)\in \partial\K$, with an angle $\theta\geq 2(\pi-\arctan\rho^{-1})>\pi$. To avoid painful arguments for a non-interesting case, let us assume that all the angles $\theta_{k}$ are greater than $\pi$, which means that $\K$ is assumed to be locally convex near the corners. For a technical reason that will be explained in due course, we also avoid the cusps and we finally assume
\begin{description}
\item[\namedlabel{H3}{\rm\bf(H3)}]
$\theta_{k}\in (\pi,2\pi)$ for all $k=1,\dots,N$.
\end{description}
 Hence, the domain considered in this paper is the exterior of holes distributed on the unit segment:
\[
\Omega^\eps =\R^2 \setminus \bigcup_{i=1}^{N_\eps}\K_{i}^\eps,
\]
with $\K_{i}^\varepsilon:=\K_{i,1}^\varepsilon$ and $z_{i}^\varepsilon:=z_{i,1}^\varepsilon$ defined in \eqref{domain1}-\eqref{domain2}.

Concerning the initial data, we consider as usual an initial vorticity independent of $\varepsilon$, compactly supported $\omega_0 \in L^\infty_{c}(\R^2)$ and we define the unique continuous initial velocity $u_0^\eps$ associated to $\omega_{0}$ through the following div-curl problem:
\begin{equation*}
\begin{split}
&\div u_0^\eps=0 \quad \textrm{in}\quad \Omega^\eps,\quad \curl u_0^\eps=\omega_0\quad \textrm{in}\quad \Omega^\eps,\quad u_0^\eps\cdot n=0\quad \textrm{on}\quad \partial\Omega^\eps,\\
& \lim_{|x|\to \infty}u_0^\eps=0,\quad \int_{\partial \K^\eps_i} u_0^\eps\cdot \tau ds=0,\quad \textrm{for all}\quad i.
\end{split}
\end{equation*}
The last condition means that the initial circulations around the holes are zero.

For fixed $\eps$, \cite{GV-L} establishes the existence of a global weak solution $(u^\varepsilon, \omega^\varepsilon)$ to the Euler equations in $\Omega^\varepsilon$ such that
\begin{description}
\item[\namedlabel{P}{\rm\bf(P)}] $\| \omega^\varepsilon\|_{L^\infty(\R^+;L^1\cap L^\infty(\Omega^\varepsilon))}\leq \| \omega_{0}\|_{L^1\cap L^\infty(\Omega^\varepsilon)}$ and the circulations around the holes remain zero.
\end{description}
We denote by $(u, \omega)$ the unique global weak solution to the Euler equations in the whole space with the initial data $(u_0, \omega_0)$, where $u_{0}$ is the solution of
\begin{equation*}
\div u_0=0 \quad \textrm{in}\quad \R^2,\quad \curl u_0=\omega_0\quad \textrm{in}\quad \R^2, \lim_{|x|\to \infty}u_0=0.
\end{equation*}

Now, we are in the position to state our main result:
\begin{theorem}\label{main thm}
Assume $\K$ verifies \ref{H1}-\ref{H3}. Let $\omega_0\in L^\infty_{c}(\R^2)$ and $(u^\eps, \omega^\eps)$ be a global weak solution (verifying \ref{P}) to the Euler equations \eqref{Euler} on $\Omega^\eps$ with initial vorticity $\omega_0|_{\Omega^\eps}$ and initial circulations 0 around the inclusions. If $|\eps \ln d_\eps|\to 0$ as $(\varepsilon,d_{\varepsilon})\to (0,0)$, then $u^\eps \to u$ strongly in $L^2_{\rm loc}(\R^+\times \R^2)$ and $\omega^\eps \rightharpoonup \omega$ weak $*$ in $L^\infty(\R^+\times \R^2)$.
\end{theorem}
Here, we have extended $(u^\varepsilon,\omega^\varepsilon)$ by zero inside the holes. For any $\beta>0$, we note that the case $d_{\varepsilon}=\varepsilon^\beta$ satisfies $|\eps \ln d_\eps|\to 0$. Even if this result is a natural extension of the previous works \cite{BLM,LM}, the proof is not obvious. Indeed, it was listed in \cite{LM} as an open question because the conformal mapping is not regular if the boundary admits a corner. The behavior of solutions of elliptic problems with the respect to the boundary regularity is an important research focus in analysis of PDE's-as testified by the huge literature (see e.g. \cite{Kenig} for results when $\partial\Omega$ is Lipschitz, and \cite{Kondra,Grisvard,Mazya} in domains with corners)- and also in complex analysis (see \cite{Pomm}). The authors of this article have already used such a theory to prove in \cite{LMW} the uniqueness of the Euler solutions if the domain has some corners whose angles are less or equal than $\pi/2$. The main idea here is to use this analysis to adapt a key proposition of \cite{LM} which will give automatically Theorem~\ref{main thm}.

Unfortunately, the permeability result when $|\eps \ln d_\eps |\to \infty$ and the case of a holes distributions in both directions cannot be achieved with this argument (see the final remark for technical details).

The rest of this article is divided in three sections. In the next part, we give the proposition concerning the behavior of the conformal mapping which will be the key to extend the analysis performed in \cite{BLM,LM}. We also state the new estimate for the cell problem and briefly recall how it implies Theorem~\ref{main thm}. Section~\ref{sec:prop} is dedicated to the proof of this new estimate. In the last section, we adapt an argument performed in \cite{ADL} and recently revisited in \cite{HLW}, to state a new stability estimate, with a precise rate, before $\omega$ reaches the segment. Namely, we will prove the following theorem.

\begin{theorem}\label{main thm2}
Assume $\K$ verifies \ref{H1}-\ref{H3}. Let $\omega_0\in C^1_{c}(\R^2\setminus ([0,1]\times\{0\}))$ and $(u^\eps, \omega^\eps)$ be a global weak solution (verifying \ref{P}) to the Euler equations \eqref{Euler} on $\Omega^\eps$ with initial vorticity $\omega_0|_{\Omega^\eps}$ and initial circulations 0 around the inclusions. Let $T>0$ and $K_{T}$ be a compact subset of $\R^2\setminus ([0,1]\times\{0\})$ such that $\supp \omega(t,\cdot) \subset K_{T}$ for all $t\in [0,T]$. If $|\eps \ln d_\eps|\to 0$, then there exist $\varepsilon_{T}>0$ and $C_{T}>0$ (depending only on $T$, $K_{T}$ and $\omega_{0}$) such that 
\[
\| \omega^{\varepsilon} - \omega \|_{L^\infty([0,T]\times \R^2)} \leq C_{T} \Big(d_\eps + \eps |\ln d_\eps |\Big)^{\frac12}, \quad \forall \varepsilon \leq \varepsilon_{T}.
\]
Moreover, for any $K$ compact subset of $\R^2\setminus([0,1]\times [-\varepsilon_{T},\varepsilon_{T}])$, there exists $C_{T,K}>0$ (depending only on $T$, $K_{T}$, $K$ and $\omega_{0}$) such that
\[
\| u^{\varepsilon}(t,\cdot) - u(t, \cdot) \|_{L^{\infty}(K)} \leq C_{T,K} \Big(d_\eps + \eps |\ln d_\eps |\Big)^{\frac12}, \quad \forall t\in [0,T], \ \forall \varepsilon \leq \varepsilon_{T}.
\]
\end{theorem}
Let us note that it is obvious that for any $T$, $\omega$ is compactly supported on $[0,T]$. So the main constraint is that this theorem holds true until $\omega$ reaches the segment. Of course, there is many cases where the 2D Euler vorticity never meets the segment, and then our estimates are global in time in these cases. As we notice in the core of the proof, this is related to a stability estimate of the Lagrangian trajectories associated to $u^\varepsilon$ and $u$. Such an estimate is usually obtained by using a $C^1$ (or log-lipschitz) norm of $u^\varepsilon - u$ which sounds very hard to obtain in the vicinity of the porous medium.

We also mention that we can write the last statement in Theorem~\ref{main thm2} as follows: 
\newline \noindent``{\it Moreover, for any $K$ compact subset of $\R^2\setminus ([0,1]\times\{0\})$, there exists $\varepsilon_{T,K},C_{T,K}>0$ (depending only on $T$, $K_{T}$, $K$ and $\omega_{0}$) such that}
\[
\| u^{\varepsilon}(t,\cdot) - u(t, \cdot) \|_{L^{\infty}(K)} \leq C_{T,K} \Big(d_\eps + \eps |\ln d_\eps |\Big)^{\frac12}, \quad \forall t\in [0,T], \ \forall \varepsilon \leq \varepsilon_{T,K}.\text{''}
\]

\section{Conformal mapping and new cell estimate}

 In this section, we bring together several arguments coming from \cite{Lacave,LMW} concerning the behavior of conformal mapping in domains with corners, and from \cite{BLM,LM} concerning the proof of Theorem~\ref{main thm} from a permeability proposition (see later Proposition~\ref{prop:key1}). This proposition is independent of the Euler motion but gives the key estimate of the cell problem for tangent divergence free vector fields.

 \subsection{Conformal mapping}\label{sec-biholo}

 Let $\Tc: \ \K^c \to \R^2\setminus \overline{B(0,1)}$ be the unique biholomorphism such that $\Tc(\infty)=\infty$ and $\Tc'(\infty)\in \R^+_{*}$, which means that there exists a bounded holomorphic function $h$ on $\K^c$ such that
\begin{equation}\label{Tinf}
 \Tc (z) = \beta z + h(z)
\end{equation}
 for some $\beta\in \R^+_{*}$. Up to a vertical translation and considering that the corner is located at $(1,h_{0})$ instead to $(1,0)$, we assume without any loss of generality that a small neighborhood of zero is included in $\K$.

This conformal mapping is of class $C^k$ up to the boundary if $\partial\Omega$ is of class $C^{k,\alpha}$, for $\alpha\in (0,1)$. The boundary does not verify this regularity assumption in domains with corners, and we collect in the following proposition the properties of $\Tc$ that we will use later.

\begin{proposition}\label{thm:conform}
Assume that $\partial\K$ verifies \ref{H1} and \ref{H3}. Let $\delta_0:=\frac 16 \min_{i\neq j}\{|x_i-x_j|, |\Tc(x_i)-\Tc(x_j)|\}$. Then there exists $M\geq 1$ depending on $\K$ such that
\begin{itemize}
\item $\Tc$ and $\Tc^{-1}$ extend continuously up to the boundary ;
\item for all $x\in \K^c \setminus \bigcup_{k=1}^N B(x_k, \delta_0)$ and $y\in \overline{B(0,1)}^c \setminus \bigcup_{k=1}^N B(\Tc(x_k), \delta_0)$, we have
\[
M^{-1} \leq |D\Tc(x)|\leq M, \quad M^{-1} \leq |D\Tc^{-1}(y)|\leq M \, ;
\]
\item for any $k=1,2,\cdots, N$ and all $x\in \K^c\cap B(x_k, \delta_0)$ and $y\in \overline{B(0,1)}^c \cap B(\Tc(x_k), \delta_0)$, we have
\begin{align*}
& M^{-1}|x-x_k|^{\pi/\theta_k-1}\leq |D\Tc(x)| \leq M |x-x_k|^{\pi/\theta_k-1},\\
& M^{-1}|y-\Tc(x_k)|^{\theta_k/\pi-1}\leq |D\Tc^{-1}(y)| \leq M |y-\Tc(x_k)|^{\theta_k/\pi-1}\, ;
\end{align*}
\item for all $x,y\in \K^c$, we have
\begin{equation*}
|\Tc(x)-\Tc(y)|\leq M\max \{ |x-y|^{\mu},\quad |x-y| \},
\end{equation*}
where $\mu=\min_{k}\frac{\pi}{\theta_{k}}\in (\frac12,1)$ ;
\item for all $x,y\in \overline{B(0,1)}^c$, we have
\begin{equation*}
|\Tc^{-1}(x)-\Tc^{-1}(y)|\leq M |x-y| .
\end{equation*}
\end{itemize}
\end{proposition}

\begin{proof}
Because of $\Tc$ is a Riemann mapping and $\partial \K\in C^{0,\alpha}$, the first bullet point can be directly obtained. 

Here, the main job is to study the behavior of $\Tc$ near corners. We consider the corner at $x_1=(1,0)$. First, we define a straighten mapping: $\varphi_1:=(z-x_1)^{\frac \pi{\theta_1}}$. It is easy to verify that $\varphi_1$ is injective and continuous on $\K^c \cap B(x_1, 2\delta_1)$ with some small data $\delta_1$.

Next, we define $D_1\subseteqq \overline{B(0,1)}^c$ be a $C^\infty$ Jordan domain such that
\[
\K^c \cap B(x_1,\delta_1) \subset \Tc^{-1}(D_1)\subset \K^c \cap B(x_1, 2\delta_1)
\]
and $g_1:D_1\to B(0,1)$ be a Riemann mapping. Define $\Omega_1:=\Tc^{-1}(D_1)$, which is $C^{1,\alpha}$ except at $x_1$, and $\widetilde{\Omega}_1:=\varphi_1(\Omega_1)$, which is $C^{1,\alpha}$ (for more details about localization and straightening, we refer to the proof of \cite[Theorem 3.9]{Pomm}).

Based on the above notations, we define a Riemann mapping $f_1=\varphi_1\circ \Tc^{-1}\circ g_1^{-1}: B(0,1)\to \widetilde{\Omega}_1$. By Kellogg-Warschawski Theorem (see \cite[Theorem 3.6]{Pomm}), we have
\[
C^{-1}_1\leq |f'_1(\zeta)| \leq C_1, \quad \forall \zeta\in \overline{B(0,1)}.
\]
On the other hand, $g_1^{-1}$ is a Riemann mapping which satisfies the same property, hence
\[
\widetilde{C}^{-1}_1\leq |(\varphi_1\circ \Tc^{-1})'(\zeta)| \leq \widetilde{C}_1, \quad \forall \zeta\in \overline{D}_1,
\]
which implies that
\[
\frac{\theta_1}{\pi \widetilde{C}_1} |\Tc^{-1}(y)-x_1|^{-\pi/\theta_1+1} \leq |(\Tc^{-1})'(y)|\leq \frac{\theta_1\widetilde{C}_1} {\pi } |\Tc^{-1}(y)-x_1|^{-\pi/\theta_1+1},\quad \forall y\in D_{1}.
\]
Thus, we get that
\[
\frac{\pi}{\theta_1 \widetilde{C}_1} | x-x_1|^{\pi/\theta_1-1} \leq |\Tc'(x)|\leq \frac{\pi \widetilde{C}_1} {\theta_1} | x-x_1|^{\pi/\theta_1-1},\quad \forall x\in \Omega_1.
\]

For any $x\in \K^c \cap B(x_1,\delta_1)$, we look for a smooth path $\gamma$ in $\K^c \cap B(x_1,\delta_1)$ joining $x_1=(1,0)$ and $x$. Due to the definition of a corner \ref{H1}, we state that there exists an angle $\theta$ such that the segment $[(1-\delta_1, 0),(1,0)]\subset \mathcal{R}(\theta)\K$, where $\mathcal{R}(\theta)$ is the rotation of angle $\theta$ around $x_{1}$ (considering $\delta_{1}$ slightly smaller if necessary). For instance, $\theta=-\lim_{s\to 0,s>0} {\rm Angle}((1,0),\Gamma'(s_k-s))+\frac{\theta_{1}} 2$ holds. Choosing $\delta_{1}$ smaller if necessary, this rotation is needed to state that the regularity of $\partial \K$ away the corner implies that, for any $(a,b)\in \mathcal{R}(\theta)K^c \cap B(x_{1}, \delta_{1})$, the segments $[(a,b),(1,b)]$ and $[(1,b),(1,0))$ belongs to $\mathcal{R}(\theta)\K^c \cap B(x_{1}, \delta_{1})$. Hence we denote by $(a,b)$ the coordinates of $\mathcal{R}(\theta)x$ and we define $\gamma$ on $[0,1]$ as
\begin{equation}\label{gamma}
\tilde \gamma(t)=((a-1) t^B +1 , b t) , \quad \gamma = \mathcal{R}(-\theta)\tilde \gamma,
\end{equation}
where $B\geq 1$ is chosen large enough such that $\tilde \gamma \subset \mathcal{R}(\theta)\K^c \cap B(x_1,\delta_1)$. Of course, if the segment $(x_{1},x]\subset \K^c \cap B(x_1,\delta_1)$, then we choose $B=1$ (which means that $\gamma$ is the segment).
Integrating on the curve $\gamma$, we get
\begin{align*}
|\Tc(x)- \Tc(x_1)|
\leq& \int_{0}^1 |\Tc'(\gamma(t))\gamma'(t)|dt\leq C\int_{0}^1 (| a-1| t^B + |b| t)^{\pi/\theta_1-1}(| a-1|B t^{B-1} + |b| ) dt\\
\leq& C\int_{0}^1 \Big((| a-1| t^B )^{\pi/\theta_1-1} | a-1|B t^{B-1} + ( |b| t)^{\pi/\theta_1-1} |b| \Big)\, dt\\
\leq& C\int_{0}^1\Big( (| a-1| )^{\pi/\theta_1} B t^{B\pi/\theta_1-1} + |b|^{\pi/\theta_1} t^{\pi/\theta_1-1}\Big)\, dt\\
\leq& C\Big(( |a-1|)^{\pi/\theta_1}+ |b|^{\pi/\theta_1} \Big) \leq C |x-x_1|^{\pi/\theta_1},\quad \forall x\in \K^c \cap B(x_1,\delta_1),
\end{align*}
where we have used $\pi/\theta_{1}<1$. In the above estimate, $C$ is independent of $a,b,B$, hence of $x$.

If we choose $D_1$ convex, by considering the segment $[y, \Tc(x_{1})]$ we also obtain that
\[
|\Tc^{-1}(y)-x_1|=|(\varphi_1\circ \Tc^{-1})(y)-(\varphi_1\circ \Tc^{-1})(\Tc(x_1))|^{\theta_1/\pi}\leq C |y-\Tc(x_1)|^{\theta_1/\pi},\quad \forall y\in D_1,
\]
which implies that
\[
|x-x_1|^{\pi/\theta_1}\leq C |\Tc(x)-\Tc(x_1)|,\quad \forall x\in \Omega_1.
\]
The two previous inequalities give the conclusion for the estimates of $(\Tc^{-1})'$ in the neighborhood of $\Tc(x_{1})$.

These inequalities yield the claims in the second bullet point for $k=1$, and similarly for any $k=2,\dots,N$. The claims in the first bullet point are also obtained by Kellogg-Warschawski Theorem, by considering a smooth domain $D_0\subseteq \overline{B(0,1)}^c$ such that 
\[
\K^c\setminus \bigcup_{k=1}^N B(x_{k},\delta_{k}) \subset \mathcal{T}^{-1}(D_{0}) \subset \K^c\setminus \bigcup_{k=1}^N B(x_{k},\delta_{k}/2),
\]
$g_1:D_0\to \overline{B(0,1)}^c$ be a Riemann mapping and $\varphi_0(z):=z$.

\medskip

We now focus on the third bullet point. We now claim that $\K^c$ is {\it a-quasiconvex} for some $a\geq 1$, that is, for any $x,y\in\K^c$ there exists a rectifiable path $\gamma$ joining $x,y$ and satisfying
$\ell(\gamma) \leq a |x-y|$.
This follows from \ref{H1} and \ref{H3} because $\partial K$ is a piecewise $C^1$ Jordan curve with no interior cusp and hence a quasidisc (see, e.g., \cite{Gustafsson}), and Ahlfors shows in \cite{Ahlfors} that in 2D, we have
\[
\partial \K \text{ is a quasidisk} \Longleftrightarrow \K^c \text{ is quasiconvex}.
\]

Hence, for any $x,y\in\K^c$, let us consider such a path $\gamma$. Then we decompose the path as $\gamma=\gamma_{0}\cup_{k} \gamma_{k}$ where $\gamma_{k}=\gamma \cap B(x_{k},\delta_{k})$ and $\gamma_{0}=\gamma \cap ( \R^2\setminus \cup_{k} B(x_{k},\delta_{k}))$. Up to shorten $\gamma$, it is clear that $\gamma$ can intersect $\partial B(x_{k},\delta_{k})$ only twice, or once (which means that $x$ or $y$ belongs to $B(x_{k},\delta_{k}))$ or never (which means that the curve avoids this disk or that $\gamma\subset B(x_{k},\delta_{k})$.

In any of these three cases, we only need to show that for $\bar x,\bar y \in \K^c \cap \overline{B(x_{k},\delta_{k})}$ and $\gamma_{k}$ a path in $\K^c \cap B(x_{k},\delta_{k})$ between these two points, then
\begin{equation}\label{interm}
\Big| \int_{0}^1 \Tc'(\gamma_{k}(t))\gamma_{k}'(t)\,dt \Big| \leq C |\bar x- \bar y |^{\pi/\theta_{k}},
\end{equation}
because it will imply
\begin{align*}
 |\Tc(x)- \Tc(y)|
\leq& \Big| \int_{0}^1 \Tc'(\gamma(t))\gamma'(t)dt \Big| \leq M \ell (\gamma_{0}) + \sum _{k }C \ell (\gamma_{k})^{\pi/\theta_{k}} \leq M \ell (\gamma) + C \sum _{k }\ell (\gamma)^{\pi/\theta_{k}}\\
\leq & C \max \{ |x-y|^{\mu},\quad |x-y| \},
\end{align*}
where $\mu = \min_{k}\frac\pi{\theta_{k}}$.

As $\int_{0}^1 \Tc'(\gamma_{k}(t))\gamma_{k}'(t)\, dt=\Tc(\bar y)-\Tc(\bar x)$ does not depend on the path, we choose another curve. If the segment $[\bar x,\bar y]\subset K^c \cap \overline{B(x_{k},\delta_{k})}$, then we consider this segment. If not, we consider a curve as \eqref{gamma}, i.e. on the form
\[
 \gamma(t)=\mathcal{R}(\theta)(A t^B, C t+\tfrac{\delta_{k}}B)+x_{k} \in \K^c \cap B(x_k,\delta_k),\quad \forall t\in [t_{0},t_{1}],
\]
where $B\geq 1$, $A,C\in \R$ and $\mathcal{R}(\theta)$ is a rotation around $x_{k}$ of angle $\theta$ such that the segment $[(-\delta_1, 0)+x_{k},x_{k}]\subset \mathcal{R}(\theta)\K$ (considering $\delta_{k}$ slightly smaller if necessary). To find such a curve, it may be possible to consider $t_{0}<0$ and $t_{1}>0$. Hence, repeating the computations below \eqref{gamma}, we get \eqref{interm} with $C$ independent of $\bar x$, $\bar y$ and $\gamma$.

\medskip

The last bullet point is much easier to prove, because it is clear that $\overline{B(0,1)}^c$ is $\frac\pi2$-quasiconvex and the two first bullet points imply that $D \Tc^{-1}$ is uniformly bounded.

This ends the proof.
\end{proof}

 \subsection{The cell estimate and Theorem~\ref{main thm}}\label{sec:2.2}

For $\varepsilon>0$ fixed, it was proved in \cite{GV-L} that the Euler equation \eqref{Euler} has a global weak solution
 \begin{equation*}
 u^\eps \in L^\infty(\R^+; L^2_{\rm loc}(\overline{\Omega_\eps})) \quad \textrm{and}\quad \omega^\eps \in L^\infty(\R^+; L^1\cap L^\infty(\Omega_\eps))
 \end{equation*}
which satisfies \ref{P}. We refer to \cite{GV-L} for the definition of weak-circulation on irregular domains. However, in the domains considered here, there are some extra regularities which come from the div-curl problem in domains with corners. We can deduce that $u^\varepsilon$ has a trace which is integrable. This implies that the weak-circulation coincides with the standard circulation and the conservation then reads as
 \[
 \int_{\K_{i}^\varepsilon}u^\varepsilon(t,\cdot)\cdot \tau \, ds=0\quad \text{ for a.e. } t\in \R^+.
 \]
 We refer to \cite[Lemma 2.7 and (2.21)]{Lacave} and \cite[Step 4 of Section 2.3]{LM} for more details.

 So the idea in \cite{BLM} is to compare $u^\eps$ which satisfies
 \begin{gather*}
\div u^\eps =0 \text{ in } \Omega^{\varepsilon},\quad \curl u^\eps =\omega^\eps \text{ in } \Omega^{\varepsilon}, \quad u^\eps\cdot n =0 \text{ on } \partial\Omega^{\varepsilon}\\
\lim_{|x|\to\infty}u^\eps(t,x)=0,\quad \oint_{\partial \K_{i}^{\varepsilon}} u^\eps\cdot \tau\, ds=0 \text{ for all }i,
\end{gather*}
and $K_{\R^2}[\omega^\eps](t,x):=\frac1{2\pi}\int_{\R^2}\frac{(x-y)^\perp}{|x-y|^2}\omega^\varepsilon(t,y)\,dy$ which verifies
 \begin{gather*}
\div K_{\R^2}[\omega^\eps] =0 \text{ in } \Omega^{\varepsilon},\quad \curl K_{\R^2}[\omega^\eps] =\omega^\eps \text{ in } \Omega^{\varepsilon}, \\
\lim_{|x|\to\infty}K_{\R^2}[\omega^\eps](t,x)=0,\quad \oint_{\partial \K_{i}^{\varepsilon}} K_{\R^2}[\omega^\eps]\cdot \tau\, ds=0 \text{ for all }i.
\end{gather*}
As the only difference is the tangency condition, the trick is to introduce, for any function $f\in L^\infty_{c}(\overline{\Omega^\eps})$, an explicit approximate solution $v^\eps[f]$ such that
\begin{equation}\label{eq.veps}
\div v^\eps[f] =0 \text{ in } \Omega^{\varepsilon}, \quad v^\eps[f]\cdot n =0 \text{ on }\partial \Omega^{\varepsilon}, \quad \lim_{|x|\to\infty}v^\eps[f](t,x)=0
\end{equation}
which is close to $K_{\R^2}[f]$ in the $L^2$ norm. The main proposition, which is proved in the next section, reads as follows.

\begin{proposition} \label{prop:key1} {\bf (Permeability)}
Assume that $\K$ verifies \ref{H1}-\ref{H3}. For any $f\in L^\infty_{c}(\overline{\Omega^\eps})$ there exists $v^\eps[f]$ satisfying \eqref{eq.veps} such that
\begin{equation*}
\|K_{\R^2}[f] - v^\eps [f] \|_{L^2(\Omega^\eps)}\leq C \| f \|_{L^1\cap L^\infty} \Big(d_\eps + \eps |\ln d_\eps |\Big)^{\frac12},
\end{equation*}
with $C$ independent of $f$ and $\eps$.
 \end{proposition}

 In the rest of this section, we repeat quickly why this proposition implies Theorem~\ref{main thm}. For more details, we refer to \cite[Section 2.2]{LM}.

\noindent {\bf Step 1: uniform $L^2$ estimate for $u^\eps-K_{\R^2}[\omega^\eps]$.}

For each time, we remark that $u^\eps - v^\eps[\omega^\eps]$ and $K_{\R^2}[\omega^\eps]-v^\eps[\omega^\eps]$ are divergence free, tend to zero when $|x|\to \infty$, have the same curl and circulations around $\K_{i}^\eps$ for all $i$. The only difference is that $u^\eps - v^\eps[\omega^\eps]$ is tangent to the boundary of $\Omega^\eps$, which implies that it is the Leray projection of $K_{\R^2}[\omega^\eps]-v^\eps[\omega^\eps]$. Therefore, by orthogonality of this projection in $L^2$ together with triangle inequality, we have
\[
\| u^\eps-K_{\R^2}[\omega^\eps] \|_{L^2(\Omega^\eps)} \leq \| u^\eps-v^\eps[\omega^\eps] \|_{L^2(\Omega^\eps)} + \| v^\eps[\omega^\eps] -K_{\R^2}[\omega^\eps] \|_{L^2(\Omega^\eps)} \leq 2 \| v^\eps[\omega^\eps] -K_{\R^2}[\omega^\eps] \|_{L^2(\Omega^\eps)}.
\]
Under the estimate of $\|\omega^\eps\|_{L^1\cap L^\infty}$ \ref{P}, Proposition~\ref{prop:key1} gives that
\[
\| u^\eps-K_{\R^2}[\omega^\eps] \|_{L^2(\Omega^\eps)} \to 0 \text{ as }\eps\to 0 \quad \text{uniformly in time.}
\]
Recalling the standard estimate for the Biot-Savart kernel:
\begin{equation}\label{BS}
 \| K_{\R^2}[f] \|_{L^\infty(\R^2)}
\leq \Big\|\frac{1}{2\pi} \int_{\R^2}\frac{|f(y)|}{|x-y|} \, d y\Big\|_{L^\infty(\R^2)}
\leq C \|f\|_{L^1(\R^2)}^{1/2} \|f \|_{L^\infty(\R^2)}^{1/2},
\end{equation}
and the fact that $\| \text{\Large $\mathds{1}$}_{\R^2\setminus \Omega^\eps} \|_{L^2} \to 0$ (because $\eps \to 0$), we infer that
\begin{equation}\label{conv:u-K}
 u^\eps-K_{\R^2}[\omega^\eps] \to 0 \quad \text{strongly in }L^\infty(\R^+;L^2(\R^2)),
\end{equation}
where we have extended $u^\eps$ by zero inside the holes.

\noindent {\bf Step 2: compactness for the vorticity.}

Thanks to the uniform estimate of $\|\omega^\eps\|_{L^1\cap L^\infty}$ \ref{P}, Banach-Alaoglu's theorem infers that we can extract a subsequence such that
\[ \omega^\eps \rightharpoonup \omega\quad\text{ weak-$*$ in }L^\infty(\R^+; L^1\cap L^\infty(\R^2)),\]
which establishes the vorticity convergence.

Next, we derive a temporal estimate, so let us fix any test function $\phi \in C_{c}^\infty((0,\infty)\times \R^2)$. In \cite{GV-L}, it was proved that $(\omega^\eps, u^\eps)$ satisfies \eqref{transport} in the weak sense for any test function compactly supported in $\Omega^\varepsilon$. However, in our case $u^\eps$ is regular enough to deduce from the tangency property that the transport equation is also verified for $\phi$ (see \cite[Prop. 2.5 $\&$ Lem. 2.6]{Lacave}):
\[
\int_{0}^\infty \int_{\R^2} \omega^\eps \partial_{t} \phi +\int_{0}^\infty \int_{\R^2} u^\eps \omega^\eps \cdot \nabla \phi =0,
\]
where we have extended\footnote{Such an extension gives a dirac mass a long the boundary when we compute $\curl u^\varepsilon$ in $\R^2$, but this relation is not needed in this paper.} $\omega^\varepsilon$ and $u^\varepsilon$ by zero in $\R^2\setminus \Omega^{\varepsilon}$. Thanks to \ref{P}, \eqref{BS} and \eqref{conv:u-K}, we state that
\[
\int_{\R^2} u^\eps \omega^\eps \cdot \nabla \phi = \int_{\R^2} (u^\eps- K_{\R^2}[\omega^\eps])\omega^\eps \cdot \nabla \phi + \int_{\R^2} K_{\R^2}[\omega^\eps]\omega^\eps \cdot \nabla \phi,
\]
 is bounded by $C\|\nabla \phi(t,\cdot)\|_{L^2}$. Hence, we have
\[\| \partial_t \omega^\eps \|_{L^\infty(\R^+;H^{-1}(\R^2))} \leq C.\]
By Lemma C.1 in \cite{Lions}, this property can be used to extract a subsequence such that
\begin{equation}\label{conv:om}
 \omega^\eps \to \omega \text{ in } C([0,T]; L^{3/2}\cap L^4(\R^2)-w) \text{ for all }T.
\end{equation}

\noindent {\bf Step 3: compactness for the velocity.}

Now, we define $u:= K_{\R^2}[\omega]$ and we use the previous steps to pass to the limit in the decomposition
\begin{equation}\label{decomp2}
u^\eps-u = (u^\eps- K_{\R^2}[\omega^\eps]) + K_{\R^2}[\omega^\eps-\omega].
\end{equation}
Thanks to \eqref{conv:u-K}, it is obvious that the first term on the right-hand side of \eqref{decomp2} converges to zero in $L^2_{\loc}(\R^+ \times \R^2)$. Concerning the second term: for $x$ fixed, the map $y\mapsto \frac{(x-y)^\perp}{|x-y|^2}$ belongs to $L^{4/3}(B(x,1))\cap L^{3}(B(x,1)^c)$, then \eqref{conv:om} implies that for all $t,x$, we have
\[
\int_{\R^2} \frac{(x-y)^\perp}{|x-y|^2} (\omega^\eps-\omega)(t,y)\, d y \to 0\quad \text{ as }\varepsilon\to 0.
\]
So, this integral converges pointwise to zero, and it is uniformly bounded by \eqref{BS} and \ref{P} with respect of $x$ and $t$. Applying the dominated convergence theorem, we obtain the convergence of $K_{\R^2}[\omega^\eps-\omega]$ in $L^2_{\loc}(\R^+ \times \R^2)$. This ends the proof of the velocity convergence.

\noindent {\bf Step 4: passing to the limit in the Euler equations.}

Finally, we verify that $(u,\omega)$ is the unique solution of the Euler equations in $\R^2$.

The divergence and curl conditions are verified by the expression: $u=K_{\R^2}[\omega]$. Next, we use that $u^\eps$ and $\omega^\eps$ satisfies $\eqref{Euler}$ in the sense of distribution in $\Omega^\eps$, that $u^\eps$ is regular enough and tangent to the boundary, to infer that for any test function $\phi\in C^\infty_c([0,\infty)\times \R^2)$, we have
\[
\int_0^\infty\int_{\R^2} \omega^\eps\partial_{t}\phi \, d x d t+\int_0^\infty\int_{\R^2} \nabla\phi \cdot u^\eps \omega^\eps \, dx d t= -\int_{\R^2}\phi(0,x)\omega_0(x)\text{\Large $\mathds{1}$}_{\Omega^\eps} d x,
\]
where we have extended $\omega^\eps$ by zero and use that $\omega^\eps(0,\cdot)=\omega_{0}\vert_{\Omega^\eps}$ (see Step 2). By passing to the limit as $\varepsilon\to 0$, thanks to the strong-weak convergence of the pair $(u^\eps,\omega^\eps)$, we conclude that $(u,\omega)$ verifies the vorticity equation. In the whole plane, this is equivalent to state that $u$ verifies the velocity equation. As this solution is unique (Yudovich theorem), we deduce that the convergences hold without extracting a subsequence. This ends the proof of Theorem~\ref{main thm}.

\begin{remark}
 Even if we have a precise rate in Proposition~\ref{prop:key1}, for solutions in the Yudovich's class, we use Banach-Alaoglu's and Ascoli theorem which do not allow us to give a rate for $u^\varepsilon - u$. For stronger solutions, we will manage in the last section to keep this rate.
\end{remark}

\section{Permeability Proposition}\label{sec:prop}

In this section, we give the proof of Proposition~\ref{prop:key1}. The strategy of the proof is the same as \cite{BLM,LM}. The main difference is the estimates of the cell problem where we have to use that the conformal mapping is less regular when $\partial \K$ has a corner.

\subsection{Construction of the correction}

 We use the explicit formula of the Green function (with Dirichlet boundary condition) in the exterior of one simply connected compact set $\K$:
\[
G_{\K}(x,y)=\frac1{2\pi} \ln \frac{| \Tc(x) - \Tc(y)|}{|\Tc(x)-\Tc(y)^*| |\Tc(y)|},
\]
where $\Tc: \ \K^c \to \R^2\setminus \overline{B(0,1)}$ is the biholomorphism defined in Section~\ref{sec-biholo}. Above, we have denoted by
 $$y^*=\frac{y}{|y|^2}$$
 the conjugate point to $y$ across the unit circle in $\R^2$. Hence, it is verified in \cite[Section 3.1]{ILL} that the following vector field
\[
\nabla^\perp \int_{\R^2\setminus \K} G_{\K}(x,y) f(y) \, dy + \frac{\int_{\R^2\setminus \K} f}{2\pi} \nabla^\perp \ln | \Tc(x)|
\]
is divergence free, tangent to the boundary, goes to zero as $|x|\to \infty$. Moreover, its curl is equal to $f$ and the circulation around $\K$ is equal to zero.

Now we introduce a cutoff function $\varphi_{i}^\eps$ equal to $1$ close to $\K_{i}^\eps$:
\begin{equation}\label{form:varphi}
 \varphi_{i}^\eps(x):= \varphi^\eps(x-z_{i}^\eps)
\end{equation}
with $\varphi^\eps\in C^1$ such that $\varphi^\eps\equiv 1$ on $\frac\eps2 \partial\K$ and $\varphi_{i}^\eps\varphi_{j}^\eps\equiv 0$ if $i\neq j$. As we will see later, $\varphi^\eps$ will be constructed such that $\supp \varphi^\eps \subset [-\varepsilon-\frac{d_{\varepsilon}}2,\frac{\varepsilon+d_{\varepsilon}}2]\times [-\varepsilon,\varepsilon]$. If we have assumed that we also have a corner at the point $(-1,0)$ then we could construct $\varphi^\eps$ such that $\supp \varphi^\eps \subset [-\frac{\varepsilon+d_{\varepsilon}}2,\frac{\varepsilon+d_{\varepsilon}}2]\times [-\varepsilon,\varepsilon]$ and then it would be obvious that $\varphi_{i}^\eps\varphi_{j}^\eps\equiv 0$ if $i\neq j$. But constructing $\varphi^\eps$ such that the support is including in a polygon instead in a square, we can avoid this assumption (see Section~\ref{subsec:cutoff} for details).

Then, the correction is defined by
\[
v^\eps[f]:= \nabla^\perp \psi^\eps,
\]
where
\begin{align*}
\psi^\eps(x) :=& \frac1{2\pi} \Big(1- \sum_{i,j} \varphi_{i}^\eps(x) \Big) \int_{\Omega^\eps} \ln|x-y| f(y)\, d y\\
&+\frac{1}{2\pi}\sum_{i,j} \varphi_{i}^\eps (x)\int_{\Omega^\eps}{\ln}\frac{ \eps |\Tca{i}(x)-\Tca{i}(y)||\Tca{i}(x)|}{2\beta|\Tca{i}(x)-\Tca{i}(y)^*|} f(y) \, d y,
\end{align*}
with
\begin{equation}\label{def-Tie}
\Tca{i}(x) := \Tc\left(\frac{x-z_{i}^\eps}{\eps/2} \right)\ : \ (\K^\eps_{i})^c \to \R^2\setminus \overline{B(0,1)}.
\end{equation}
In the neighborhood of $\K_{i}^\eps$, this correction corresponds to the Biot-Savart law in the exterior of one obstacle, whereas, far away the porous medium, it is equal to the Biot Savart law in the whole plane $\R^2$. More precisely, we can check that $v^\eps[f]$ verifies
the following properties:
\begin{equation*}
\div v^\eps[f] =0 \text{ in } \Omega^\varepsilon, \quad v^\eps[f] \cdot n =0 \text{ on } \partial \Omega^\varepsilon, \quad \lim_{x\to \infty} |v^\eps[f](x)|=0.
\end{equation*}

We decompose $K_{\R^2}[f]-v^\eps[f]$ as
\begin{equation} \label{decompo we}
K_{\R^2}[f]-v^\eps[f] =\frac1{2\pi} \sum_{i} \nabla^\perp \varphi_{i}^\eps(x) (w_{i}^{1,\eps}+w_{i}^{2,\eps}) + \varphi_{i}^\eps(x) (w_{i}^{3,\eps}+w_{i}^{4,\eps}) ,
\end{equation}
where
\begin{equation*}\begin{split}
w_{i}^{1,\eps}(x)=& \int_{\Omega^\eps} \ln \frac{2\beta|x-y|}{\eps|\Tca{i}(x)-\Tca{i}(y)|}f(y)\, dy, \\
w_{i}^{2,\eps}(x)=& \int_{\Omega^\eps} \ln \frac{|\Tca{i}(x)-\Tca{i}(y)^*|}{|\Tca{i}(x)|}f(y)\, dy, \\
w_{i}^{3,\eps}(x)=& \int_{\Omega^\eps} \Biggl(\frac{(x-y)^\perp}{|x-y|^2}- (D\Tca{i})^T(x)\frac{(\Tca{i}(x)-\Tca{i}(y))^\perp}{|\Tca{i}(x)-\Tca{i}(y)|^2} \Biggl) f(y)\, dy, \\
w_{i}^{4,\eps}(x)=& (D\Tca{i})^T(x) \int_{\Omega^\eps} \Biggl(\frac{\Tca{i}(x)-\Tca{i}(y)^*}{|\Tca{i}(x)-\Tca{i}(y)^*|^2}- \frac{\Tca{i}(x)}{|\Tca{i}(x)|^2}\Biggl)^\perp f(y)\, dy.
\end{split}\end{equation*}

In the following subsection, we estimate $w_{i}^{k,\eps}$ on the support of $\varphi_{i}^\eps$, and next, we look for the best cutoff function $\varphi^\eps$.

\subsection{Cell problem estimates}

When $\K=\overline{B(0,1)}$, $\Tc={\rm Id}$ (so $\beta=1$) and $w_{1}^\eps=w_{3}^\eps \equiv 0$. In this case, we also have $\Tca{i}(x)-\Tca{i}(y)^* =\dfrac2\eps\Big(x-z_{i}^\eps -\eps^2 \dfrac{y-z_{i}^\eps}{4|y-z_{i}^\eps|^2}\Big)$. Except in an $\varepsilon$-neighborhood of the inclusion, we note that $\Tca{i}(y)^*$ is small compared to $\Tca{i}(x)$. Hence, we can guess that $w_{2}^\eps$ and $w_{4}^\eps$ are small. This remark is the main motivation of this decomposition, and in the following estimates, we split the integrals in two parts: a small area in the vicinity of the inclusion and the far away region where $\Tc$ behaves as $\beta\ {\rm Id}$. Due to the lost of the boundary regularity, some changes are needed compared to \cite{BLM,LM} for the estimates close to the holes.

From the definition of $ \Tca{i}$ \eqref{def-Tie}, it is clear that Proposition~\ref{thm:conform} gives
\begin{equation}\label{quasiLip}
|\Tca{i}(x) - \Tca{i}(y)| \leq C\max\{\eps^{-
\mu} |x-y|^\mu, \eps^{-1}|x-y|\}
\end{equation}
where $\mu=\min_{k}\frac{\pi}{\theta_{k}}$.
Moreover, we have that $(\Tca{i})^{-1}(y)=\frac\varepsilon2 \Tc^{-1}(y)+z_{i}^\varepsilon$ and Proposition~\ref{thm:conform} also implies
\begin{equation}\label{Lip}
\| (\Tca{i})^{-1} \|_{\mathrm{Lip}} \leq C\varepsilon .
\end{equation}

As $\Tc$ behaves at infinity as $\beta\ {\rm Id}$, it holds that for all $r>0$
\begin{equation}\label{anneau1}
 \Tca{i}\Bigl(\partial B(z^\eps_{i},r)\cap (\K^\eps_{i})^c\Bigr) \subset B\Bigl(0,C_{1}\frac r\eps\Bigr)\setminus B\Bigl(0,C_{2}\frac r\eps\Bigr)
\end{equation}
and
\begin{equation}\label{anneau2}
 (\Tca{i})^{-1}\Bigl(\partial B(0,r+1)\Bigr) \subset B\Bigl(z^\eps_{i},\eps C_{3}(r+1)\Bigr)\setminus B\Bigl(z^\eps_{i},\eps C_{4}(r+1)\Bigr).
\end{equation}
for some $C_{1},C_{2},C_{3},C_{4}$ positive numbers independent of $i, \eps$. For a proof, we refer to \cite[Lemma 2.2]{BLM} where we can easily check that we only use the fact that $\Tc$ and $\Tc^{-1}$ is continuous up to the boundary and that a small neighborhood of zero is included in $\K$ (see the beginning of Section~\ref{sec-biholo}).
\medskip

\underline {\bf Estimate of $w_{i}^{1,\eps}$.} For $x\in\supp \varphi^\eps_{i}$ fixed, we decompose the integral in two parts:
\begin{equation}\label{eq.loinpres}
\begin{split}
&\Omega^\eps_{C}:=\{y\in \Omega^\eps,\ |\Tca{i}(x)-\Tca{i}(y)|\leq \eps^{-1/4} \},\\
&\Omega^\eps_{F}:=\{y\in \Omega^\eps,\ |\Tca{i}(x)-\Tca{i}(y)|> \eps^{-1/4} \}.
\end{split}
\end{equation}

In the subdomain close to the inclusion $\Omega^\eps_{C}$ \eqref{eq.loinpres}, we set $z=\eps\Tca{i}(x)$ and we change variables $\eta=\eps \Tca{i}(y)$:
\begin{equation*}\begin{split}
\int_{\Omega^\eps_{C}} \Bigl| \ln (\eps|\Tca{i}(x)-\Tca{i}(y)|) f(y)\Bigl| \, d y
&\leq \int_{B(z, \eps^{3/4})} \Bigl| \ln |z-\eta| f((\Tca{i})^{-1}(\tfrac\eta\eps))\Bigl| \frac{ \bigl| \det D(\Tca{i})^{-1}\bigr|(\tfrac\eta\eps)}{\varepsilon^2}\, d \eta\\
&\leq \int_{B(z,\eps^{3/4})} \Bigl| \ln |z-\eta| f((\Tca{i})^{-1}(\tfrac\eta\eps))\Bigl|\tfrac14 \bigl| \det D\Tc^{-1}\bigr|(\tfrac\eta\eps) d \eta.
\end{split}\end{equation*}
Using that $D\Tc^{-1}$ and $f$ are bounded functions, we compute that:
\begin{equation*}
\int_{\Omega^\eps_{C}} \Bigl| \ln (\eps|\Tca{i}(x)-\Tca{i}(y)|) f(y)\Bigl| \, d y
\leq C\| f\|_{L^\infty} \int_{B(0, \eps^{3/4} )} \Bigl| \ln |\xi| \Bigl| \, d \xi \leq C \| f\|_{L^\infty} \eps^{3/2}|\ln \eps| .
\end{equation*}
To deal with $\ln(2\beta|x-y|)$, we remark that if $y\in \Omega^\eps_{C}$, then \eqref{Lip} gives
\[
|x-y|=|(\Tca{i})^{-1}(\Tca{i}(x)) - (\Tca{i})^{-1}(\Tca{i}(y))| \leq \eps C |\Tca{i}(x)- \Tca{i}(y)| \leq C \eps^{3/4}.
\]
So, we have
\begin{align*}
\int_{\Omega^\eps_{C}} \Bigl| \ln( {2\beta|x-y|})f(y) \Bigl|\, d y
&\leq \int_{B(x, C\eps^{3/4})} \Bigl| \ln (2\beta|x-y|) f(y) \Bigl|\, d y \\
&\leq \|f\|_{L^\infty} \int_{B(0, C\eps^{3/4})} \Bigl| \ln |2\beta \xi| \Bigl|\, d\xi
\leq C\|f\|_{L^\infty} \eps^{3/2} |\ln \eps|.
\end{align*}

In the subdomain far away from the inclusion $\Omega^\eps_{F}$ \eqref{eq.loinpres}, we have by \eqref{quasiLip}
\[
\eps^{-1/4} \leq |\Tca{i}(x) - \Tca{i}(y)| \leq C\max\{\eps^{-
\mu} |x-y|^\mu, \eps^{-1}|x-y|\}.
\]
Hence $|x-y|\geq \min\{ \frac{ \eps^{1-\frac1{4\mu}}}C, \frac{ \eps^{3/4 }}C\}= \frac{ \eps^{3/4}}C$ because $\mu=\min_i\{\frac{\pi}{\theta_i}\}<1$.

On the other hand, we use the definition of $ \Tca{i}$ \eqref{def-Tie} and the decomposition \eqref{Tinf} to write
\begin{equation}\label{eq.ln2}
\ln \frac{\eps|\Tca{i}(x)-\Tca{i}(y)|}{2\beta|x-y|}
= \ln \frac{\Bigl|2\beta (x-y) + \eps \Bigl(h\big(\frac{x-z^\eps_{i}}{\eps/2}\big)- h\big(\frac{y-z^\eps_{i}}{\eps/2}\big)\Bigr)\Bigr|}{2\beta|x-y|}.
\end{equation}
When $\varepsilon$ is small enough, we have in $\Omega^\eps_{F}$
\begin{align*}
\frac{\eps \Bigl| h\big(\frac{x-z^\eps_{i}}{\eps/2}\big)- h\big(\frac{y-z^\eps_{i}}{\eps/2}\big)\Bigl| }{2\beta|x-y|}
\leq \frac{ \eps \|h\|_{L^\infty}}{ \beta |x-y|} \leq C\varepsilon^{1/4} \leq \frac12.
\end{align*}
 We note easily that
\begin{equation}\label{est ln}
\Bigl|\ln\tfrac{|b+c|}{|b|} \Bigl| \leq 2 \tfrac{|c|}{|b|},\qquad\mbox{ if }\tfrac{|c|}{|b|}\leq \tfrac12.
\end{equation}
Applying this inequality with $c=\varepsilon\Big(h\big(\frac{x-z^\eps_{i}}{\eps/2}\big)- h\big(\frac{y-z^\eps_{i}}{\eps/2}\big)\Big)$ and $b= 2\beta(x-y)$, we compute from \eqref{eq.ln2}:
\[
\Biggl|\ln \frac{\eps|\Tca{i}(x)-\Tca{i}(y)|}{2\beta|x-y|}\Biggl|
\leq 2\frac{\eps|h(\frac{x-z^\eps_{i}}{\eps/2})-h(\frac{y-z^\eps_{i}}{\eps/2})|}{2\beta|x-y|} \leq \frac{C\eps }{ |x-y|}.
\]
Therefore, using \eqref{BS}, we obtain
\begin{align*}
\int_{\Omega^\eps_{F}} \Bigl|\ln \frac{2\beta|x-y|}{\eps|\Tca{i}(x)-\Tca{i}(y)|} f(y) \Bigl| \, d y
&\leq C\eps\int_{ \Omega^\eps}\frac{ |f(y)| }{|x-y|} \, d y\\
&\leq C\eps \| f\|_{L^\infty}^{1/2} \| f\|_{L^1}^{1/2}
\end{align*}
which allows us to conclude that
\begin{equation}\label{est:w1ij}
 \| w_{i}^{1,\eps} \|_{L^\infty (\supp \varphi^\eps_{i})} \leq C\varepsilon \| f\|_{L^1\cap L^\infty}
\end{equation}
with $C$ independent of $i, \eps$ and $f$.

\medskip

\underline {\bf Estimate of $w_{i}^{2,\eps}$.} Setting $z=\eps \Tca{i}(x)$, and changing variables $\eta = \eps\Tca{i}(y)$, we get
\begin{equation*}
w_{i}^{2,\eps}(x)= \int_{B(0,\eps)^c} \ln \frac{|z- \eps^2 \eta^*|}{|z|}f(\tfrac\eps2 \Tc^{-1}(\tfrac\eta\eps)+z^\eps_{i}) \tfrac14 |\det D\Tc^{-1}|(\tfrac\eta\eps) \, d \eta.
\end{equation*}
As mentioned in the beginning of Section~\ref{sec-biholo}, we assume, without loss of generality, that there is $\delta >0$ so that $B(0,\delta)\subset \K$. Hence, for $x\in \supp \varphi^\eps_{i}$, we note that $ |x-z^\eps_{i}| \geq\delta \eps $,
then, we deduce by \eqref{anneau1} that $|z|\geq C_{2} \delta \eps $. So, for any $\eta$ we have
\begin{equation*}
\frac{|\eps^2 \eta^*|}{|z|} \leq \frac{\eps}{C_{2}\delta |\eta|},
\end{equation*}
and infer by \eqref{est ln} (with $b=z$ and $c=-\eps^2 \eta^*$) that
\begin{equation*}
\left| \ln \frac{|z- \eps^2 \eta^*|}{|z|} \right|
\leq 2 \frac{\eps^2 |\eta^*|}{|z|}
\leq \frac{2\eps}{C_{2} \delta |\eta|}\qquad \mbox{ if }\quad \frac{\eps}{C_{2}\delta |\eta|}\leq \frac12.
\end{equation*}
Therefore, we define $R=2/(C_{2}\delta)$ and split the integral in two parts: $ B(0,R\eps)^c$ and $B(0,R\eps)\setminus B(0,\eps)$.

In the first subdomain $B(0,R\eps)^c$, we use the previous inequality to compute
\begin{equation*}\begin{split}
\Bigl|\int_{B(0,R\eps)^c}& \ln \frac{|z- \eps^2 \eta^*|}{|z|}f(\tfrac\eps2 \Tc^{-1}(\tfrac\eta\eps)+z^\eps_{i})\tfrac14 |\det D\Tc^{-1}|(\tfrac\eta\eps) \, d \eta\Bigl|\\
&\leq \frac{2\eps}{C_{2}\delta} \int_{\R^2} \frac{|f(\tfrac\eps2 \Tc^{-1}(\tfrac\eta\eps)+z^\eps_{i})| \tfrac14 |\det D\Tc^{-1}|(\tfrac\eta\eps)}{|\eta|} \, d \eta\\
&\leq C\eps \Big\| f(\tfrac\eps2 \Tc^{-1}(\tfrac\eta\eps)+z^\eps_{i}) \tfrac14 \det D\Tc^{-1}(\tfrac\eta\eps) \Big\|_{L^\infty}^{1/2}
\Big\| f(\tfrac\eps2 \Tc^{-1}(\tfrac\eta\eps)+z^\eps_{i}) \tfrac14\det D\Tc^{-1}(\tfrac\eta\eps) \Big\|_{L^1}^{1/2}\\
&\leq C\eps \|f\|_{L^\infty}^{1/2}\|f\|_{L^1}^{1/2},
\end{split}\end{equation*}
where we have applied \eqref{BS} for the function $\eta \mapsto |f(\tfrac\eps2 \Tc^{-1}(\tfrac\eta\eps)+z^\eps_{i})|\tfrac14 |\det D\Tc^{-1}|(\tfrac\eta\eps)$ at $x=0$, used that $D \Tc^{-1}$ is bounded and that $\| f(\tfrac\eps2 \Tc^{-1}(\tfrac\eta\eps)+z^\eps_{i})\tfrac14 \det D\Tc^{-1}(\tfrac\eta\eps) \|_{L^1}=\|f\|_{L^1}$ by changing variables back.

In the second subdomain $B(0,R\eps)\setminus B(0,\eps)$, we use the relation
\[
\frac{|z- \eps^2 \eta^*|}{|z|}=\frac{|\eta- \eps^2 z^*|}{|\eta|}
\]
which can be easily verified by squaring both side. As $D\Tc^{-1}$ is bounded and $\eps^2 z^*\in B(0,\eps)$, we compute
\begin{equation*}\begin{split}
\Bigl|\int_{B(0,R\eps)\setminus B(0,\eps)} \ln \frac{|z- \eps^2 \eta^*|}{|z|}f(\tfrac\eps2 \Tc^{-1}(\tfrac\eta\eps)+z^\eps_{i}) \tfrac14 |\det D\Tc^{-1}|(\tfrac\eta\eps) \, d \eta\Bigl|
&\leq 2C \| f\|_{L^\infty}\int_{B(0,(R+1)\eps)} | \ln |\eta| |\, d\eta\\
&\leq C\| f\|_{L^\infty} \eps^2 |\ln \eps|,
\end{split}\end{equation*}
which allows us to conclude that
\begin{equation}\label{est:w2ij}
 \| w_{i}^{2,\eps} \|_{L^\infty (\supp \varphi^\eps_{i})} \leq C\varepsilon \| f\|_{L^1\cap L^\infty}
\end{equation}
with $C$ independent of $i, \eps$ and $f$.

\underline {\bf Estimate of $w_{i}^{3,\eps}$ and $w_{i}^{4,\eps}$.}
 Rewriting $w_{i}^{3,\eps}$:
 \begin{equation*}\begin{split}
w_{i}^{3,\eps}(x)=& \int_{\Omega^\eps} \frac{(x-y)^\perp}{|x-y|^2} f(y)\, dy- (D\Tca{i})^T(x) \int_{\Omega^\eps} \frac{(\Tca{i}(x)-\Tca{i}(y))^\perp}{|\Tca{i}(x)-\Tca{i}(y)|^2} f(y)\, dy, \\
=:& I_1(x)+ (D\Tca{i})^T(x) I_2(x).
\end{split}\end{equation*}
By \eqref{BS}, we have that
\[
\|I_1\|_{L^\infty (\Omega^\eps)}\leq C \| f\|_{L^1\cap L^\infty}.
\]
For $I_2$, like $w_{i}^{2,\eps}$, we have
\[
I_2 ((\Tca{i})^{-1}(\tfrac z\eps))=\varepsilon\int_{B(0,\eps)^c} \frac{(z-\eta )^\perp}{|z-\eta|^2} f(\tfrac\eps2 \Tc^{-1}(\tfrac \eta\eps)+z_{i}^\varepsilon)\tfrac14 |\det D\Tc^{-1}|(\tfrac \eta\eps) \, d \eta.
\]
So with the same argument than for the first part of $w_{i}^{2,\eps}$, we use \eqref{BS} together with the bound of $ D\Tc^{-1}$ to state
\[
\|I_2\|_{L^\infty (\Omega^\eps)}\leq\varepsilon C \| f\|_{L^1\cap L^\infty}.
\]
Finally, by Proposition~\ref{thm:conform}, we obtain that
\begin{align}
 \| w_{i}^{3,\eps} \|_{L^4 (\supp \varphi_{i}^\varepsilon)} & \leq \| 1 \|_{L^4 (\supp \varphi_{i}^\varepsilon)}\|I_1\|_{L^\infty (\Omega^\eps)} + \| D\Tca{i} \|_{L^4 (\supp \varphi_{i}^\varepsilon)}\|I_2\|_{L^\infty (\Omega^\eps)} \\
 & \leq C\Big( \varepsilon^{1/4}(\varepsilon+d_\varepsilon)^{1/4} +\varepsilon^{1/2} \| D\Tc \|_{L^4 ([-2-\frac{d_{\varepsilon}}{\varepsilon} , 1+\frac{d_{\varepsilon}}{\varepsilon} ] \times [-2,2]\setminus \K)} \Big) \| f\|_{L^1\cap L^\infty}\nonumber\\
 &\leq C \varepsilon^{1/2} \Big(1 +\Big(\frac{d_{\varepsilon}}{\varepsilon}\Big)^{1/4}\Big)\| f\|_{L^1\cap L^\infty}
 \leq C \varepsilon^{1/4} \Big(\varepsilon^{1/4}+d_{\varepsilon}^{1/4}\Big)\| f\|_{L^1\cap L^\infty} \label{est:w3ij}
\end{align}
where we have used that $D\Tc$ belongs to $L^4$ close to $\partial \K$ (as $\theta_{k}<2\pi$) and goes to $\beta$ at infinity.

For $w_{i}^{4,\eps}$, we write a similar decomposition:
\[
w_{i}^{4,\eps}(x)=(D\Tca{i})^T(x)I_3(x)
\]
with
\[
I_3((\Tca{i})^{-1}(\tfrac z\eps))=\varepsilon \int_{B(0,\eps)^c}\Biggl(\frac{z-\varepsilon^2\eta^*}{|z-\varepsilon^2\eta^*|^2}- \frac{z}{|z|^2}\Biggl)^\perp f(\tfrac\eps2 \Tc^{-1}(\tfrac \eta\eps)+z_{i}^\varepsilon)\tfrac14 |\det D\Tc^{-1}|(\tfrac \eta\eps) \, d \eta.
\]
Then we follow the argument of \cite[Theorem 2.1]{ILL}, i.e. we split the integral in two parts:
\begin{itemize}
 \item in $B(0,2\eps)^c$ where $|z-\varepsilon^2\eta^*|\geq \varepsilon/2$ (recalling that $|z|\geq \varepsilon$), so using that
 \[
 \Big|\frac{a}{|a|^2}-\frac{b}{|b|^2}\Big|=\frac{|a-b|}{|a||b|}
 \]
 we need to estimate
 \[
 \varepsilon \int_{B(0,2\eps)^c}\frac{\varepsilon^2}{\frac\varepsilon2 \varepsilon} \frac{f(\tfrac\eps2 \Tc^{-1}(\tfrac \eta\eps)+z_{i}^\varepsilon)}{|\eta|}\tfrac14 |\det D\Tc^{-1}|(\tfrac \eta\eps) \, d \eta
 \]
 which is less than $C \varepsilon \| f\|_{L^1\cap L^\infty}$ (see the argument above).
\item in $B(0,2\eps)\setminus B(0,\varepsilon) $, where it is clear that
\[
\varepsilon \int_{B(0,2\eps)} \frac{1}{|z|} |f(\tfrac\eps2 \Tc^{-1}(\tfrac \eta\eps)+z_{i}^\varepsilon)|\tfrac14 |\det D\Tc^{-1}|(\tfrac \eta\eps) \, d \eta \leq \varepsilon^2 \| f\|_{L^\infty}
\]
and where we change variables $\xi=\varepsilon^2\eta^*$ in the last integral to compute
\[
\varepsilon \int_{B(0,\eps)\setminus B(0,\varepsilon/2)}\frac{(z-\xi)^\perp}{|z-\xi |^2} f(\tfrac\eps2 \Tc^{-1}( \eps\xi^*)+z_{i}^\varepsilon)\tfrac14 |\det D\Tc^{-1}|(\eps\xi^*)\frac{\varepsilon^4}{|\xi|^4} \, d \xi
\]
 which is less than $C \varepsilon \| f\|_{L^1\cap L^\infty}$ (as before, using \eqref{BS} and changing variable back to compute $\|f\|_{L^1}$).
\end{itemize}
We conclude as for $w_{i}^{3,\eps}$:
\begin{equation}\label{est:w4ij}
 \| w_{i}^{4,\eps} \|_{L^4 (\supp \varphi_{i}^\varepsilon)} \leq C \varepsilon^{1/4} \Big(\varepsilon^{1/4}+d_{\varepsilon}^{1/4}\Big)\| f\|_{L^1\cap L^\infty} .
\end{equation}

Therefore, putting together the form of $\varphi^\eps_{i}$ \eqref{form:varphi}, the decomposition \eqref{decompo we} and the estimates \eqref{est:w1ij}, \eqref{est:w2ij}, \eqref{est:w3ij} and \eqref{est:w4ij} we have obtained:
\begin{equation}\label{est:permeability}
\begin{split}
 \| K_{\R^2}[f]-v^\eps[f] \|_{L^2(\Omega^\eps)} &\leq C \|f\|_{L^1\cap L^\infty} \Bigl(\eps \| \nabla \varphi^\eps \|_{L^2} + \varepsilon^{1/4}\Big(\varepsilon^{1/4}+d_{\varepsilon}^{1/4}\Big)\| \varphi^\eps \|_{L^4} \Bigl) \sqrt{\sum_{i} 1}\\
 &\leq C \|f\|_{L^1\cap L^\infty} \Bigl(\eps^{1/2} \| \nabla \varphi^\eps \|_{L^2} + \varepsilon^{-1/4}\Big(\varepsilon^{1/4}+d_{\varepsilon}^{1/4}\Big) \| \varphi^\eps \|_{L^4} \Bigl) .
 \end{split}
\end{equation}

\subsection{Optimal cutoff function}\label{subsec:cutoff}

 So the question is to find the best $\varphi^\eps$ such that the right hand side term of \eqref{est:permeability} tends to zero.

We consider two cases.

In the regime where $d_{\varepsilon}/\varepsilon >\delta$ for some $\delta >0$, then we consider $\varphi\in C^\infty$ such that $\varphi\equiv1$ on $[-1,1]^2$ and $\varphi\equiv0$ on $[-1-\delta,1+\delta]^2$ and we set $\varphi^\eps(x):=\varphi(\frac{x}{\varepsilon/2})$. As $\supp \varphi^\eps \subset [-\frac{\varepsilon+d_{\varepsilon}}2,\frac{\varepsilon+d_{\varepsilon}}2]^2 \setminus [-\frac{\varepsilon}2,\frac{\varepsilon}2]^2$, we obviously have that $\varphi^\eps\equiv 1$ on $\frac\eps2 \partial\K$ and $\varphi_{i}^\eps\varphi_{j}^\eps\equiv 0$ if $i\neq j$. We also compute $ \| \nabla \varphi^\eps \|_{L^2} \leq C$ and $ \| \varphi^\eps \|_{L^4} \leq C \varepsilon^{1/2}$.
 In this case, \eqref{est:permeability} reads as
 \[
 \| K_{\R^2}[f]-v^\eps[f] \|_{L^2(\Omega^\eps)} \leq C \|f\|_{L^1\cap L^\infty} \Bigl(\eps^{1/2} +d_{\varepsilon}^{1/2}\Bigl)
 \]
 which implies the estimate stated in Proposition~\ref{prop:key1}.

Now, we consider the interesting regime where $d_{\varepsilon}/\varepsilon \to 0$.
The idea is to define a cutoff function which depends on the space between $\frac\eps2\K$ and $\frac\eps2\K + (d_{\varepsilon}+\varepsilon,0)$. As we do not have assumed that there is a corner at the point $(-1,0)$, we will only use that \ref{H2} implies that
\[
\{ (x_{1},x_{2}): \ x_{2}\in [-\varepsilon,\varepsilon], \ \tfrac\varepsilon2 -\rho |x_{2}| < x_{1} < \tfrac\varepsilon2 + d_{\varepsilon} \} \subset \R^2 \setminus \Big(\tfrac\eps2\K^\varepsilon \cup \{ \tfrac\eps2\K^\varepsilon + (d_{\varepsilon}+\varepsilon,0)\}\Big) ,
\]
and we split this area in two:
\begin{gather*}
 \{ (x_{1},x_{2}): \ x_{2}\in [-\varepsilon,\varepsilon], \ \tfrac\varepsilon2 -\rho |x_{2}| < x_{1} \leq \tfrac{\varepsilon + d_{\varepsilon}}2 - \tfrac\rho2 |x_{2}| \} \\
\{ (x_{1},x_{2}): \ x_{2}\in [-\varepsilon,\varepsilon], \ \tfrac{\varepsilon + d_{\varepsilon}}2 - \tfrac\rho2 |x_{2}| < x_{1} < \tfrac\varepsilon2 + d_{\varepsilon} \}.
\end{gather*}
Therefore, we construct $\varphi^\varepsilon$ such that
\begin{equation}\label{phiH1}
 \varphi^\varepsilon\equiv 1 \text{ on } \{ (x_{1},x_{2}): \ x_{2}\in [-\tfrac\varepsilon2,\tfrac\varepsilon2], \ -\tfrac{\varepsilon}2 < x_{1} < \tfrac{\varepsilon}2 - \rho |x_{2}| \}
\end{equation}
and
\begin{equation}\label{phiH2}
\varphi^\varepsilon\equiv 0 \text{ on } \R^2\setminus \{ (x_{1},x_{2}): \ x_{2}\in [-\varepsilon,\varepsilon], \ -\tfrac{\varepsilon + d_{\varepsilon}}2 - \tfrac\rho2 |x_{2}| < x_{1} < \tfrac{\varepsilon + d_{\varepsilon}}2 - \tfrac\rho2 |x_{2}| \}.
\end{equation}
Up to decrease $\rho$, we can assume that $\rho<1$, and it holds true that $\varphi^\eps\equiv 1$ on $\frac\eps2 \partial\K$ and $\varphi_{i}^\eps\varphi_{j}^\eps\equiv 0$ if $i\neq j$, and that $\supp \varphi^\eps \subset [-\varepsilon-\frac{d_{\varepsilon}}2,\frac{\varepsilon+d_{\varepsilon}}2]\times [-\varepsilon,\varepsilon]$.

Hence, we define
\[
d(x_{2}):= \tfrac{d_{\eps}}2+ \tfrac\rho2 |x_{2} |
\]
which verifies
\[
d(x_{2})=\Big( \tfrac{\varepsilon + d_{\varepsilon}}2 - \tfrac\rho2 |x_{2}| \Big)-\Big( \tfrac{\varepsilon}2 - \rho |x_{2}| \Big)
\]
and
\[
d(x_{2})=-\tfrac{\varepsilon}2-\Big(-\tfrac{\varepsilon + d_{\varepsilon}}2 - \tfrac\rho2 |x_{2}|\Big).
\]
Let $\varphi\in \mathcal{C}^\infty(\R)$ be a positive non-increasing function such that $\varphi(s)=1$ if $s\leq 0$ and $\varphi(s)=0$ if $s\geq 1$. We finally introduce:
\[
\varphi^\varepsilon(x)=\varphi\Big(\frac{2|x_{2}|-\varepsilon}\varepsilon\Big)\Big[ 1-\varphi\Big(\frac{( \tfrac{\varepsilon + d_{\varepsilon}}2 - \tfrac\rho2 |x_{2}| ) -x_{1}}{d(x_{2})}\Big)-\varphi\Big(\frac{x_{1}-(-\tfrac{\varepsilon + d_{\varepsilon}}2 - \tfrac\rho2 |x_{2}|)}{d(x_{2})}\Big)\Big] .
\]
As $\rho<1$, we have that $-\tfrac{\varepsilon}2 \leq \tfrac{\varepsilon}2 - \rho |x_{2}| $ for all $x_{2}\in [-\varepsilon,\varepsilon]$ and we can check that this $\varphi^\varepsilon$ verifies \eqref{phiH1}-\eqref{phiH2}.

It is not a problem that the cut-off function $\varphi^\varepsilon \notin C^1$ because the set where it is not derivable is negligible. All we want is $\| \nabla \varphi^\eps \|_{L^2}$ and $\| \varphi^\eps \|_{L^4}$.

Since $\varphi^\varepsilon(x)\in [0,1]$ for all $x\in \R^2$, it is clear from \eqref{phiH2} that
\[
\| \varphi^\varepsilon \|_{L^4}\leq C( \eps(\eps + d_{\eps}))^{\frac14} \leq C\eps^{1/2},
\]
as we are in the regime $d_{\varepsilon}<\varepsilon$.

Next, we see that for all $x$ we have
\[
|\nabla\varphi^\varepsilon(x)| \leq \frac{C}\varepsilon + \frac{C}{d(x_{2})}.
\]
where we have used that on the support of $\varphi'\Big(\frac{( \tfrac{\varepsilon + d_{\varepsilon}}2 - \tfrac\rho2 |x_{2}| ) -x_{1}}{d(x_{2})}\Big)$ we have $| ( \tfrac{\varepsilon + d_{\varepsilon}}2 - \tfrac\rho2 |x_{2}| ) -x_{1}|\leq d(x_{2})$. So we compute:
\begin{equation*}
\begin{split}
\| \nabla \varphi^\varepsilon \|_{L^2}\leq& C\Big(\frac{2\eps(\varepsilon+d_{\eps})}{\varepsilon^2}+ 4\int_{0}^{\varepsilon}\int_{0}^{d(x_{2})} \frac1{d(x_{2})^2} \, dx_{1}dx_{2} \Big)^{1/2}\leq C\Big(1+\frac{d_{\eps}}{\eps}+ 4\int_{0}^{\varepsilon} \frac1{d(x_{2})} \, dx_{2} \Big)^{1/2}\\
& \leq C\Big(1+\frac{d_{\eps}}{\eps}+ \int_{0}^{\varepsilon} \frac{dx_2}{d_{\eps}+ \rho x_2} \Big)^{1/2}
\leq C\Big(1+\frac{d_{\eps}}{\eps}+ \ln{\frac{ d_\eps+ \rho\eps}{d_\eps}} \Big)^{1/2}.
\end{split}
\end{equation*}
In the case $d_{\eps} \leq \eps$, this gives
\[
\| \nabla \varphi^\varepsilon \|_{L^2} \leq C \Big(1 +\frac{d_{\eps}}{\eps}+ \ln{\frac{\eps }{d_\eps}} \Big)^{1/2} \leq C \Big(1 + | \ln d_\eps| \Big)^{1/2}.
\]

Thus, \eqref{est:permeability} reads as
 \[
 \| K_{\R^2}[f]-v^\eps[f] \|_{L^2(\Omega^\eps)} \leq C \|f\|_{L^1\cap L^\infty} \Bigl(d_{\varepsilon}^{1/2} + \eps^{1/2} | \ln d_\eps|^{1/2}\Bigl)
 \]
which ends the proof of Proposition~\ref{prop:key1}.

\section{Estimates with precise rate for stronger solutions}

This section is dedicated to the proof of Theorem~\ref{main thm2}, then $\omega_{0}$ is now assumed to be smoother and compactly supported in $K_{T}$, as $\omega(t, \cdot)$ for all $t\in [0,T]$. We mainly follow the proof of \cite[Sec. 7.2]{ADL}.

For $\varepsilon_{1}>0$ small enough, let us introduce two intermediate compact subsets $K_{1}$, $K_{2}$ such that
\[
K_T\Subset K_1\Subset K_2 \Subset \Omega^{\varepsilon}, \quad \text{for all $\varepsilon\leq \varepsilon_{1}$,}
\]
then we define $T_{\varepsilon}\in (0, T]$ such that $\omega^\varepsilon$ stays compactly supported in $K_{1}$:
\begin{equation*}
 T_{\varepsilon}:= \sup_{\tilde T\in [0,T]} \Big\{ \tilde T, \ \supp \omega^\varepsilon (t,\cdot)\subset K_{1} \ \forall t\in [0,\tilde T] \Big\}.
\end{equation*}
By local regularity argument, $u^\varepsilon$ is continuous in $K_{1}$ and transports the vorticity, so we state that $T_{\varepsilon}>0$ and that there are only two possibilities:
\begin{enumerate}
 \item $T_{\varepsilon}=T$, hence $\supp \omega^\varepsilon(t,\cdot)\subset K_{1}$ for all $t\in [0,T]$ ;
 \item $T_\varepsilon<T$, hence $\overline{\supp \omega^\varepsilon(T_{\varepsilon},\cdot)} \cap \partial K_{1} \neq \emptyset$. \label{pageTN}
\end{enumerate}
In the sequel of this section, we will derive uniform estimates for all $t\in [0,T_{\varepsilon}]$ (where the support of $\omega^{\varepsilon}$ is included in $K_{1}$), and we will conclude by a bootstrap argument that (2) cannot happen if $\varepsilon$ is small enough, which will imply that the estimates hold true on $[0,T]$.

\bigskip

For Yudovich solutions, local elliptic estimates on $K_{1}$ allow us to define uniquely the Lagrangian flow: for any $x\in \supp \omega_{0}$, there exists $t(x)>0$ and a unique curve $ X^\varepsilon(\cdot,x)\in W^{1,\infty}([0,t(x)))$ such that $X^\varepsilon(t,x)\in K_{1}$ for each $t\in[0,t(x))$,
\[
X^\varepsilon(t,x) = x + \int_{0}^t u^\varepsilon(s,X^\varepsilon(s,x))\, ds \qquad \forall t\in [0,t(x)),
\]
as well as $X^\varepsilon(t(x),x)\in \partial K_{1}$ if $t(x)<T_{\varepsilon}$. As $u^\varepsilon$ is uniformly (in time) log-Lipshitz on $K_{1}$, we obtain
\begin{equation}\label{flow}
 \frac{d}{dt} X^\varepsilon(t,x) = u^\varepsilon(t,X^\varepsilon(t,x)) \qquad \text{for a.e. } t\in [0,t(x)).
\end{equation}

When the velocity $u^\varepsilon$ is not globally regular, it is not clear that the solution $\omega^\varepsilon$ of the linear transport equation (for $u^\varepsilon$ given) constructed in \cite{GV-L} coincides with the solution which is constant along the characteristics $X^\varepsilon(t,\cdot)_\#\omega_0$. 
For domains with corners, one may show that $u^\varepsilon$ is even smoother and that the vorticity is a renormalized solution of the transport equation \eqref{transport} in the sense of DiPerna-Lions (see \cite[Lem. 2.7 $\&$ 2.8]{Lacave}). By uniqueness for linear transport equations \cite{DiPernaLions}, we deduce that 
 $$\omega^\varepsilon(t)=X^\varepsilon(t,\cdot)_\#\omega_0, \quad \text{for a.e. }t\geq 0 $$
in the sense that for a.e. $t\geq 0$ we have $\int_{\Omega^\varepsilon} \omega^\varepsilon(t,x)\varphi(x)\,dx=\int_{\Omega^\varepsilon} \omega_0(x)\varphi(X^\varepsilon(t,x))\,dx$ for all $\varphi\in C_c(\Omega^\varepsilon)$. Due to the definition of $T_{\varepsilon}$, we deduce that $t(x)=T_{\varepsilon}$ for all $x\in \supp \omega_{0}$.

After redefining $\omega^\varepsilon$ on a set of measure zero, this becomes $\omega^\varepsilon(t,x)=\omega_{0}((X^\varepsilon)^{-1}(t,\cdot))(x)$. Uniform boundedness of $u^\varepsilon$ on $K_{1}$ now yields $\omega^\varepsilon\in C([0,T_{\varepsilon}];L^1(\Omega^\varepsilon))$. It is then not hard to show that $u^\varepsilon$ is continuous\footnote{The velocity is log-lipschitz uniformly in $t$ (the constant depends only on $\|\omega^\varepsilon \|_{L^1\cap L^\infty}$, see e.g. \cite[App. 2.3]{MarPul}). From the div-curl problem verified by $u(t_1,\cdot)-u(t_2,\cdot)$, it is possible to show $| u(t_1,x)-u(t_2,x) | \leq C \| \omega(t_1,\cdot )-\omega(t_2, \cdot) \|_{L^1}^{1/2}\| \omega(t_1,\cdot)-\omega(t_2,\cdot ) \|_{L^\infty}^{1/2}$ (see, e.g., \cite[Theo. 4.1]{ILL}) which then implies the continuity with respect to $t$.} on $[0,T_{\varepsilon}]\times K_{1}$, which also means that \eqref{flow} holds for all $(t,x)\in[0,T_{\varepsilon}]\times \supp \omega_{0}$. For more details about renormalized solutions, we refer, for instance, to \cite[Prop. 4.1]{LMW} or \cite[Sect. 3.2]{LZ}. For $\omega_{0}\in C^1$, it is then possible to show from the formula $\omega^\varepsilon(t,x)=\omega_{0}((X^\varepsilon)^{-1}(t,\cdot))(x)$ that $\omega^\varepsilon$ belongs to $C^1_{c}([0,T_{\varepsilon}]\times K_{1})$.

We also define the flow associated to $(u,\omega)$: $(t,x)\mapsto X(t,x)$ on $\R^+\times \R^2$ by
\begin{equation}\label{def-X}
 \left\{
\begin{array}{l}
 \frac{\partial X}{\partial t}(t,x) = u(t,X(t,x)), \\
 X(0,x)=x,
\end{array}
 \right.
\end{equation}
and we will use that $\omega$ is constant along these trajectories: $\omega(t,X(t,x))=\omega_{0}(x)$.

\subsection{Stability estimate for velocities}\label{sec-stab-vel}

The first step of our proof is to derive a uniform estimate of $u-u^\varepsilon$ in $[0,T_{\varepsilon}]\times K_{1}$. By Proposition~\ref{prop:key1} together with \ref{P} we get easily by orthogonality of the Leray projector (see Step 1 in Section~\ref{sec:2.2}) that for all $t\in [0,T_{\varepsilon}]$:
\[
 \| (u^\eps-K_{\R^2}[\omega^\eps])(t,\cdot) \|_{L^2(\Omega^\eps)} 
 \leq 2 \| v^\eps[\omega^\eps(t,\cdot)] -K_{\R^2}[\omega^\eps(t,\cdot)] \|_{L^2(\Omega^\eps)} 
 \leq C(d_\eps + \eps |\ln d_\eps | )^{\frac12} ,
\]
so we deduce by harmonicity (as $u^\eps-K_{\R^2}[\omega^\eps]$ is curl and divergence free)
\[
 \| (u^\eps-K_{\R^2}[\omega^\eps])(t,\cdot) \|_{L^\infty(K_{1})} 
 \leq C \| (u^\eps-K_{\R^2}[\omega^\eps])(t,\cdot) \|_{L^2(K_{2})} 
 \leq C(d_\eps + \eps |\ln d_\eps | )^{\frac12} .
\]
Using also \eqref{BS} and the compact support of $\omega$ and $\omega^\varepsilon$, we finally get for all $\varepsilon\leq \varepsilon_{1}$ and $t\in [0,T_{\varepsilon}]$
\begin{align}
 \| (u^\eps- u)(t,\cdot) \|_{L^\infty(K_{1})}
 \leq & \| (u^\eps-K_{\R^2}[\omega^\eps])(t,\cdot) \|_{L^\infty(K_{1})} + \|K_{\R^2}[\omega^\eps-\omega](t,\cdot) \|_{L^\infty(K_{1})} \nonumber\\
 \leq &C\Big( (d_\eps + \eps |\ln d_\eps | )^{\frac12} + \| (\omega^\eps-\omega)(t,\cdot) \|_{L^\infty(\R^2)} \Big).\label{stab-vel}
\end{align}

Moreover, using the standard elliptic estimate (see, e.g., \cite[Lem. 7.2]{ADL})
\[
\| \nabla K_{\R^2} [f] \|_{L^\infty(\R^2)} \leq C \Big(1+ \| f \|_{L^1\cap L^\infty(\R^2)}+ \| f \|_{L^\infty(\R^2)} \ln(1+ \| \nabla f \|_{L^\infty(\R^2)})\Big)\\
\]
we again obtain by harmonicity that
\begin{align}
 \| \nabla u^\eps(t,\cdot) \|_{L^\infty(K_{1})}
 \leq & \| (u^\eps-K_{\R^2}[\omega^\eps])(t,\cdot) \|_{L^2(K_{2})} + \|\nabla K_{\R^2}[\omega^\eps](t,\cdot) \|_{L^\infty(K_{1})} \nonumber\\
 \leq &C\Big( 1 + \ln(1+ \| \nabla \omega^\eps(t,\cdot) \|_{L^\infty(\R^2)}) \Big), \qquad \forall \varepsilon\leq \varepsilon_{1},\ \forall t\in [0,T_{\varepsilon}]. \label{est-vel}
\end{align}

\subsection{Uniform $C^1$ estimates for vorticities}

Differentiating the vorticity equation, we get for $i=1,2$:
\[
\partial_{t} \partial_{i} \omega^\varepsilon + u^\varepsilon\cdot \nabla \partial_{i} \omega^\varepsilon=-\partial_{i} u^\varepsilon \cdot \nabla \omega^\varepsilon,
\]
hence
\[
\partial_{i} \omega^\varepsilon(t,X^\varepsilon(t,x)) = \partial_{i} \omega_{0}(x) -\int_{0}^t (\partial_{i} u^\varepsilon \cdot \nabla \omega^\varepsilon)(s,X^\varepsilon(s,x))\, ds.
\]
As $X^\varepsilon(t,x)\in K_{1}$ for all $(t,x)\in [0,T_\varepsilon]\times \supp \omega_{0}$ and as we have a bound for $\| \nabla u^\varepsilon \|_{L^\infty([0,T_\varepsilon]\times K_{1})}$ (see \eqref{est-vel}), we get that
\[
\| \nabla \omega^\varepsilon(t,\cdot) \|_{L^\infty(\R^2)} \leq \| \nabla \omega_{0} \|_{L^\infty(\R^2)} + C \int_{0}^t \| \nabla \omega^\varepsilon(s,\cdot) \|_{L^\infty(\R^2)}\ln(2+ \| \nabla \omega^\varepsilon(s,\cdot) \|_{L^\infty(\R^2)}) \, ds.
\]
Gronwall's lemma allows us to conclude the following estimate for the vorticity:
\begin{equation}\label{est-vort}
 \| \nabla \omega^\varepsilon(t,\cdot) \|_{L^{\infty}(\R^2)} \leq C , \qquad \forall \varepsilon\leq \varepsilon_{1},\ \forall t\in [0,T_\varepsilon].
\end{equation}

\subsection{Stability estimate for vorticities}

Subtracting the vorticity equations, we can write
\begin{gather*}
 \partial_{t}(\omega-\omega^\varepsilon) + u\cdot \nabla (\omega-\omega^\varepsilon)=-(u-u^\varepsilon)\cdot \nabla \omega^\varepsilon,\\
 \partial_{t}(\omega-\omega^\varepsilon) + u^\varepsilon\cdot \nabla (\omega-\omega^\varepsilon)=-(u-u^\varepsilon)\cdot \nabla \omega
\end{gather*}
which imply that
\begin{gather*}
(\omega-\omega^\varepsilon)(t,X(t,x)) =- \int_{0}^t \Big( (u-u^\varepsilon)\cdot \nabla \omega^\varepsilon\Big)(s,X(s,x))\, ds,\\
(\omega-\omega^\varepsilon)(t,X^\varepsilon(t,x)) =- \int_{0}^t \Big( (u-u^\varepsilon)\cdot \nabla \omega\Big)(s,X^\varepsilon(s,x))\, ds.
\end{gather*}
As the support of $(\omega-\omega^\varepsilon)(t,\cdot)$ is included in $X(t,\supp \omega_{0}) \cup X^\varepsilon(t,\supp \omega_{0})$, we use \eqref{stab-vel} and \eqref{est-vort} to write
\[
 \| ( \omega-\omega^\varepsilon)(t,\cdot) \|_{L^{\infty}(\R^2)} \leq C\Big( (d_\eps + \eps |\ln d_\eps | )^{\frac12} + \int_{0}^t \|( \omega-\omega^\varepsilon)(s,\cdot)\|_{L^\infty(\R^2)} \, ds \Big), \qquad \forall \varepsilon\leq \varepsilon_{1},\ \forall t\in [0,T_\varepsilon].
\]

Therefore, Gronwall's lemma gives
\begin{equation}\label{stab-vort}
 \| ( \omega-\omega^\varepsilon)(t,\cdot) \|_{L^{\infty}(\R^2)} \leq C(d_\eps + \eps |\ln d_\eps | )^{\frac12} , \qquad \forall \varepsilon\leq \varepsilon_{1},\ \forall t\in [0,T_\varepsilon],
 \end{equation}
and \eqref{stab-vel} becomes
\begin{equation}\label{stab-vel2}
 \| (u - u^\varepsilon)(t,\cdot) \|_{L^{\infty}(K_1)} \leq C(d_\eps + \eps |\ln d_\eps | )^{\frac12} , \qquad \forall \varepsilon \leq \varepsilon_{1},\ \forall t\in [0,T_\varepsilon].
 \end{equation}

\subsection{Stability estimate for trajectories}

From the definition of the trajectories \eqref{flow}-\eqref{def-X} and repeating the decomposition of Section~\ref{sec-stab-vel}, we compute
\begin{align*}
 \partial_{t} |(X^\varepsilon-X)(t,x) |^2 &\leq 2|(X^\varepsilon-X)(t,x) | \Big( |(u^\varepsilon-u)(t,X^\varepsilon(t,x)) | + |u (t,X^\varepsilon(t,x) ) -u(t,X(t,x)) | \Big)\\
 &\leq C|(X^\varepsilon-X)(t,x) | \Big( (d_\eps + \eps |\ln d_\eps | )^{\frac12} + |(X^\varepsilon-X)(t,x) | \Big), \qquad \forall \varepsilon\leq \varepsilon_{1},\ \forall t\in [0,T_\varepsilon],
\end{align*}
where we have used \eqref{stab-vel2} and that $u\in C^1([0,T]\times \R^2)$. We deduce again by Gronwall's lemma that 
\begin{equation}\label{stab-traj}
 | (X^\varepsilon-X)(t,x) | \leq C (d_\eps + \eps |\ln d_\eps | )^{\frac12} , \qquad \forall \varepsilon\leq \varepsilon_{1},\ \forall t\in [0,T_\varepsilon],\ \forall x\in \supp \omega_{0}.
 \end{equation}

\subsection{Bootstrap argument and conclusion}

To summarize, for $\omega_{0}, T, K_{T}$ given, we fix $K_{1}, K_{2}, \varepsilon_{1}$, so there exists $C>0$ such that the estimates \eqref{stab-vort}-\eqref{stab-traj} are valid for all $\varepsilon\leq \varepsilon_{1}$ and $t\in [0,T_\varepsilon]$. Now we choose $\varepsilon_{2}\leq \varepsilon_{1}$ such that $C (d_\eps + \eps |\ln d_\eps | )^{\frac12} \leq \frac12 {\rm d}(\partial K_{T},\partial K_{1})$ for all $\varepsilon\leq \varepsilon_{2}$. As $X(t,x)\in K_{T}$ for all $(t,x)\in [0,T]\times \supp \omega_{0}$, we conclude from \eqref{stab-traj} that the situation (2) in Page \pageref{pageTN} is impossible. This allows us to conclude that $T_\varepsilon=T$ for all $\varepsilon\leq \varepsilon_{2}$ and that \eqref{stab-vort}-\eqref{stab-traj} are valid for all $\varepsilon\leq \varepsilon_{2}$, $t\in [0,T]$.

In Section~\ref{sec-stab-vel}, replacing $K_{1}$ by any compact subset $K$ of $\R^2\setminus([0,1]\times [-\varepsilon_{2},\varepsilon_{2})]$, and using \eqref{stab-vort}, we get easily that \eqref{stab-vel2} is valid if we replace $K_{1}$ by $K$. This ends the proof of Theorem~\ref{main thm2}.

\subsection*{Remark on open questions (2)-(3) listed in page \pageref{openquestions}}

These two questions cannot be easily solved by adapting the analysis developed in \cite{LM}. Therein, the key to get the impermeability was to compute the area between two holes. However, with a corner we can follow line by line Section 4.1 in \cite{LM} with $\gamma=0$ and we cannot hope better than 
\[
\mathcal{A}^\varepsilon(s) \leq C (\varepsilon s^2 + d_{\varepsilon}s)
\]
which means that
\[
\frac{\mathcal{A}^\varepsilon(s)}{(\varepsilon s)^2} \leq C \Big( \frac1\varepsilon + \frac{d_{\varepsilon}}{\varepsilon^2 s})
\]
is never small. This prevents us to estimate the fluid flux passing through the segment in the $L^2$ framework as it was done in \cite{LM}.

\subsection*{Acknowledgements}
This work has been supported by the ``Sino-French Research Program in Mathematics'' (SFRPM), which made several visits between the authors possible. C.~L. would like to acknowledge the hospitality and financial support of Peking University to conduct part of this work. C.~L. is partially supported by the CNRS (program Tellus) and by the Agence Nationale de la Recherche: Project IFSMACS (grant ANR-15-CE40-0010) and Project SINGFLOWS (grant ANR-18-CE40-0027-01). C.~W. is partly supported by NSF of China under Grant 11701016.

\def\cprime{$'$}

\end{document}